\documentclass[11pt,psamsfonts]{amsart}
\usepackage{amsmath}
\usepackage{amsthm}
\usepackage{amssymb}
\usepackage{amscd}
\usepackage{amsfonts}
\usepackage{amsbsy}
\usepackage{graphicx}
\usepackage[dvips]{psfrag}
\usepackage{array}
\usepackage{color}
\usepackage{overpic}
\usepackage{epstopdf}
 
\usepackage{hyperref}
\usepackage{float}
\usepackage{pict2e}

\newcolumntype{L}{>{\displaystyle}l}
\newcolumntype{C}{>{\displaystyle}c}
\newcolumntype{R}{>{\displaystyle}r}

\newcommand{\R}{\ensuremath{\mathbb{R}}}
\newcommand{\N}{\ensuremath{\mathbb{N}}}

\newcommand{\Z}{\ensuremath{\mathbb{Z}}}

\newcommand{\CC}{\mathcal{C}}
\newcommand{\CA}{\ensuremath{\mathcal{A}}}
\newcommand{\CR}{\ensuremath{\mathcal{R}}}

\newcommand{\CF}{\ensuremath{\mathcal{F}}}
\newcommand{\CG}{\ensuremath{\mathcal{G}}}
\newcommand{\CH}{\ensuremath{\mathcal{H}}}
\newcommand{\CM}{\ensuremath{\mathcal{M}}}
\newcommand{\CT}{\ensuremath{\mathcal{T}}}
\newcommand{\CO}{\ensuremath{\mathcal{O}}}
\newcommand{\CS}{\ensuremath{\mathcal{S}}}

\newcommand{\s}{\ensuremath{\mathbb{S}}}

\newcommand{\ov}{\overline}
\newcommand{\la}{\lambda}
\newcommand{\g}{\gamma}
\newcommand{\G}{\Gamma}
\newcommand{\T}{\theta}

\newcommand{\al}{\alpha}

\newcommand{\si}{\sigma}

\newcommand{\sgn}{\mathrm{sign}}
\newcommand{\de}{\delta}

\newcommand{\fix}{\mathrm{Fix}}

\def\p{\partial}
\def\e{\varepsilon}

\newtheorem {theorem} {Theorem} %[section]
\newtheorem {proposition} [theorem] {Proposition}
\newtheorem {corollary} [theorem] {Corollary}
\newtheorem {lemma} [theorem] {Lemma}

\newtheorem {remark}[theorem]{Remark}

\newtheorem {mtheorem} {Theorem}

\usepackage{mathpazo}

\begin{document}

\title[Periodic perturbations of a two-fold cycle in reversible Filippov systems]
{Study of periodic orbits in periodic perturbations of\\ planar reversible Filippov systems\\ having a two-fold cycle}

\author[D.D.Novaes, T.M. Seara, M.A. Teixeira and I.O. Zeli]
{Douglas D. Novaes$^{1}$, Tere M. Seara$^2$, Marco A. Teixeira$^{1}$ and Iris O. Zeli$^{3}$}

\address{$^1$ Departamento de Matem\'{a}tica, Universidade
Estadual de Campinas, Rua S\'{e}rgio Buarque de Holanda, 651, Cidade
Universit\'{a}ria Zeferino Vaz, 13083--859, Campinas, SP, Brazil}
\email{ddnovaes@unicamp.br}
\email{teixeira@unicamp.br}

\address{$^2$ Departament de Matem\`{a}tica Aplicada I,
Universitat Polit\`{e}cnica de Catalunya, Diagonal 647, 08028 Barcelona, Spain} \email{tere.m-seara@upc.edu}

\address{$^3$ Departamento de Matem\'{a}tica, Instituto Tecnol\'ogico de Aeron\'autica, Pra\c{c}a Marechal Eduardo Gomes, 50, Vila das Ac\'acias, 12228--900, S\~ao Jos\'e dos Campos, SP, Brazil}
\email{iriszeli@ita.br}



\maketitle

\begin{abstract}
We study the existence of periodic solutions in a class of planar Filippov systems obtained 
from non-autonomous periodic perturbations of reversible piecewise smooth differential systems. It is assumed that the unperturbed system presents a simple two-fold cycle, which is characterized by a closed trajectory connecting a visible  two-fold singularity to itself. 
It is shown that under certain generic conditions the perturbed system has
sliding and crossing periodic solutions. 
In order to get our results, Melnikov's ideas were applied together with  tools from the geometric singular perturbation theory. 
Finally, a study of a perturbed piecewise Hamiltonian model is performed.
\end{abstract}

\section{Introduction}\label{intro}

Over the last decade, the theory of non-smooth dynamical systems has been developed at a very fast pace, with growing importance at the frontier between mathematics, physics, engineering, and the life sciences (see, for instance,  \cite{CRM,BBCK,Mike18,Var}, and references therein). The study of such systems goes back to the work of Andronov et. al \cite{AndronovEtAl66}  in 1937. A rigorous mathematical formalization of this theory was provided by Filippov \cite{F} in 1988, who used the theory of differential inclusions for establishing the definition of trajectory for non-smooth differential systems. Nowadays, such systems are called Filippov systems.

In 1981, motivated by the work of Ekeland \cite{E} on discontinuous Hamiltonian vector fields, Teixeira \cite{T} studied generic singularities of refractive non-smooth vector fields. It was performed a qualitative analyses of two-fold singularities appearing in these systems. Later, the generic classification of such singularities has been approached in several works \cite{GST, K,KRG}. 

Recently, many efforts have been dedicated to understand some typical global minimal sets in Filippov systems (see, for instance, \cite{AJMT,AGN,LH,NPV,NT,NTZ}) . In particular, Novaes et al. \cite{NTZ} studied the unfolding of a {\it Simple Two-Fold Cycle} (see Figure \ref{piecewise}) inside the class of autonomous planar Filippov systems.
A {\it Simple Two-Fold Cycle} is characterized by a closed trajectory connecting a two-fold singularity to itself and having a non-constant first return map defined in one side of the cycle (see Figure \ref{piecewise}).

The present study focuses on understanding how a simple two-fold cycle unfolds under small periodic perturbations. More specifically, 
 we are mainly concerned with sliding and crossing periodic solutions bifurcating from a simple two-fold cycle of  a $R$-reversible planar Filippov system periodically perturbed. By $R$-reversibility of a Filippov system,
\begin{equation}
\label{eq0}
Z_0(x,y)=\left\{\begin{array}{l} F^+(x,y) \quad \textrm{if} \quad y>0,\vspace{0.2cm}\\
F^-(x,y) \quad \textrm{if}\quad y<0,
\end{array}\right.
\end{equation}
we mean $F^+(x,y)=-RF^-R(x,y),$ where $R: \R^2 \rightarrow \R^2$ is an involution for which $y=0$ is the set of fixed points (see \cite{JTT}). Here, we shall consider $R(x,y)=(x,-y)$. For this involution, the $R$-reversibility implies that $F^+(x,y)=(-F_1(x,-y),F_2(x,-y))$ and $F^-(x,y)=(F_1(x,y),F_2(x,y))$.  As a consequence of the $R$-reversibility, a Simple Two-Fold Cycle $\CS$ of \eqref{eq0} is always a boundary of a period annulus $\mathcal{A}$ of crossing periodic solutions. Here, we shall assume that $\CS$ encloses such a period annulus  (see Figure \ref{piecewise}).

As examples of piecewise smooth differential systems satisfying the  hypotheses above, we have the following one-parameter family of piecewise Hamiltonian differential systems, 
\begin{equation}\label{example0}
Z_0^\al(x,y)= \big((1,x^2-\al),(-1,x^2-\al)\big), \quad \al>0,
\end{equation} with Hamiltonian function given by
\[
H(x,y)=|y|-\dfrac{x^3}{3}+\al x.
\]
The vector field $Z_0^\al$ contains a simple two-fold cycle $\CS$ connecting the visible two-fold singularity $(\sqrt{\al},0)$ to itself.  This cycle encloses an annulus $\CA$ fulfilled with crossing periodic orbits (see Figure \ref{piecewise}). 

\begin{figure}[H]
\begin{center}
\begin{overpic}[width=7cm]{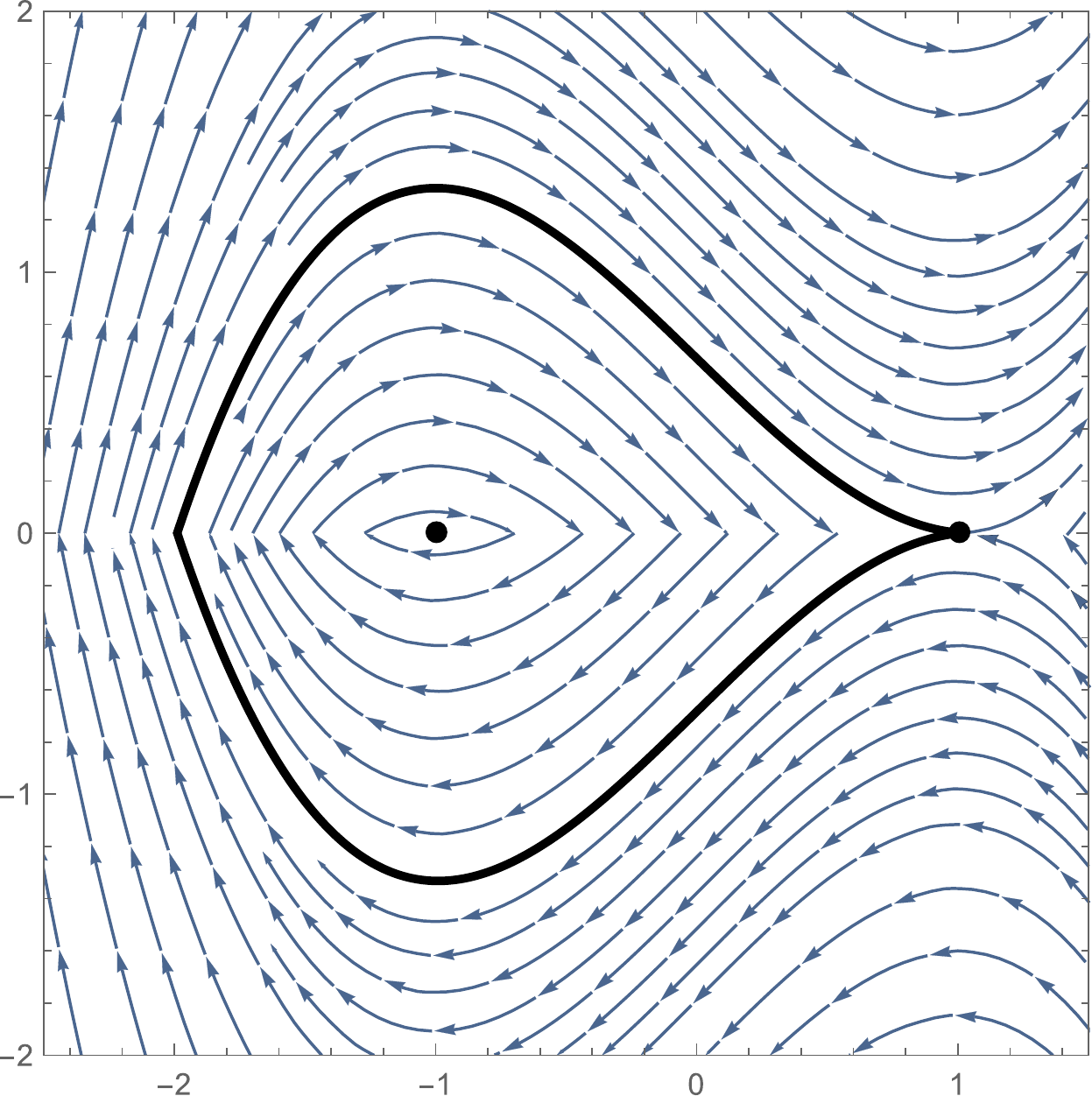}
%\begin{overpic}[grid,tics=5,width=12cm]{homoclinic.pdf}
\end{overpic}
\end{center}

\caption{Phase space of the piecewise smooth differential system $(\dot x,\dot y)^T=Z_0^\al(x,y)=\big((1,x^2-\al),(-1,x^2-\al)\big),$ for $\al=1$.  In general, the points $(-\sqrt{\al},0)$ and $(\sqrt{\al},0)$ are the invisible and visible two-fold singularities, respectively. The bold line represents the  simple two-fold cycle $\CS$, which encloses a period annulus  $\mathcal{A}$ of crossing periodic orbits.}
\label{piecewise}
\end{figure}

In our setting, the construction of a suitable displacement function and its related Melnikov function are the central mechanisms behind our study. As it is fairly known in Melnikov theory, the existence of periodic solutions bifurcating from a period annulus is associated with simple zeros of a certain bifurcation function, called {\it Melnikov function}. Such a function is obtained through the analysis of the perturbed  system using its regular dependence with respect to the perturbation parameter. 
Indeed, in the smooth case, the displacement function (equivalently, the Poincar\'{e} Map) is smooth in the parameter of perturbation. Consequently, the Melnikov function is obtained by expanding the displacement function in Taylor series.  The same procedure has  ben used in some non-smooth systems to study crossing periodic solutions (see, for instance, \cite{BastosEtAl19,BattelliFeckan08,CruzEtAl19,GouveiaEtAl16,GHS,LlibreEtAl15} and the references therein). However, such an approach fails when facing sliding dynamics, which appears, for instance, in the unfolding of two-fold singularities.
Thus, the main novelty of this study consists in the analysis of crossing and sliding periodic solutions bifurcating from a simple two-fold cycle $\CS$ which, as noticed above, is the boundary of a period annulus $\CA,$ in the reversible context. 
The developed procedure for the detection of sliding periodic solutions is rather different, because
  regular perturbations of a Filippov system produce singular perturbation problems in the sliding dynamics. Accordingly,  tools from singular perturbation theory must be employed. We shall see that, although unexpected, the same Melnikov function, obtained by the former classical approach for detecting crossing periodic solutions bifurcating from $\CA,$ also plays an important role in the study of the sliding periodic solutions.

We emphasize that the above mentioned theoretical aspects has been the main motivation behind our study. To the best of our knowledge, non-smooth models of real phenomena exhibiting two-fold cycles are not known so far. Nevertheless, such as our initial example \eqref{example0}, this kind of cycle can be easily found in piecewise mechanical systems. 

This paper is organized as follows. 
First, in Section \ref{prel}, we present the basic notions and results needed to state our main Theorems. 
More specifically, in Section \ref{pds}, we recall the basic definitions about Filippov systems, 
and in Section \ref{bcr} we give some basic concepts and results concerning the reversible unperturbed problem. 
In Section \ref{MR}, we state our main results, Theorems \ref{ta} and \ref{tb}. 
They deal with periodic non-autonomous perturbations of $R$--reversible piecewise smooth differential systems admitting a simple two-fold cycle.
More specifically, we provide a Melnikov function which determines the existence of crossing and sliding periodic solutions for such systems. 
In Theorem \ref{ta}, it is shown that this function determines the existence of  crossing periodic solutions bifurcating from orbits of the period annulus $\mathcal{A}$. In Corollary \ref{ca}, we also consider autonomous perturbations. 
In Theorem \ref{tb}, it is shown that the same Melnikov function also determines, with additional hypotheses, 
the existence of both sliding and crossing periodic solutions bifurcating from the  simple two-fold cycle $\CS$. 
In Section \ref{exam}, we apply our results to study periodic non-autonomous perturbations of the piecewise Hamiltonian differential system \eqref{example0}. 
Finally, Section \ref{proof} is devoted to prove our main results. Some conclusion remarks and further directions are provided in Section \ref{sec:conc}.

\section{Basic concepts and preliminary results}\label{prel}

In this section, we recall the basic concepts and definitions from the theory of non-smooth dynamical systems as well as some preliminary 
results needed to state our main theorems.

\subsection{Filippov systems}\label{pds} The content of this Section is standard and can be found in several other works (see for instance \cite{Fi}). 

Let $U$ be an open bounded subset of $\R^2$. 
We denote by $\CC^r(\ov U,\R^2)$ the set of all $\CC^r$ vector fields $X:\ov U\rightarrow \R^n$. 
Given $h:\ov U\rightarrow\R$ a differentiable function having $0$ as a regular value, we denote by $\Omega_h^r(\ov U,\R^2)$ the space of piecewise smooth differential systems
$Z$ in $\R^2$ such that
\begin{equation}\label{omega}
Z(x,y)=\left\{\begin{array}{l}
X^+(x,y),\quad\textrm{if}\quad h(x,y)>0,\vspace{0.1cm}\\
X^-(x,y),\quad\textrm{if}\quad h(x,y)<0,
\end{array}\right.
\end{equation}
with $X^+,X^- \in \CC^r(\ov U,\R^2)$. As usual, system \eqref{omega} is denoted by $Z=(X^+,X^-)$ and the switching surface $h^{-1}(0)$ by $\Sigma$.

The points on $\Sigma$ where both vectors fields $X^+$ and $X^-$ simultaneously point outward or inward from $\Sigma$ define, respectively, 
the {\it escaping} $\Sigma^e$ or {\it sliding} $\Sigma^s$ regions, and the interior of its complement in $\Sigma$ defines the {\it crossing region} $\Sigma^c$. 
The complementary of the union of those regions  are the {\it tangency} points between $X^+$ or $X^-$ with $\Sigma$.

The points in $\Sigma^c$ satisfy $X^+ h(p)\cdot X^-h(p) > 0$, where $Xh$ denotes the derivative of the function $h$ in the direction of the vector $X$, 
that is $Xh(p)=\langle \nabla h(p), X(p)\rangle$. 
The points in $\Sigma^s$ (resp. $\Sigma^e$) satisfy $X^+h(p)<0$ and $X^-h(p) > 0$ (resp. $X^+h(p)>0$ and $X^-h(p) < 0$). 
Finally, the tangency points of $X^+$ (resp. $X^-$) satisfy $X^+h(p)=0$ (resp. $X^-h(p)=0$). 
For points $p\in\Sigma^s\cup\Sigma^e$, we define the {\it sliding vector field}
\begin{equation*}\label{slisys}
\widetilde Z(p)=\dfrac{X^- h(p) X^+(p)-X^+ h(p) X^-(p)}{X^- h(p)- X^+h(p)}.
\end{equation*}

A tangency point $p\in\Sigma$ is called a {\it visible fold} of 
$X^+$ $($resp. $X^-)$ if $(X^+)^2h(p)>0$ $($resp. $(X^-)^2h(p)<0)$.  Analogously, reversing the inequalities, we define an {\it invisible fold}. 

\subsection{Preliminary results}\label{bcr}

Consider the involution $R(x,y)=(x,-y)$  and denote by $\fix(R)=\{(x,0), x\in\R \}$ its set of fixed points. 
For a $\CC^2$ function $F:D\rightarrow\R^2$, defined on an open bounded subset $D$ of $\R^2$, 
we consider the following $R$-reversible discontinuous piecewise smooth differential system with two zones separated by the straight line $\Sigma=\fix(R)$,
\begin{equation}\label{ups1}
(x',y')^T=Z_0(x,y)=\left\{\begin{array}{l} F^+(x,y) \quad \textrm{if} \quad y>0,\vspace{0.2cm}\\
F^-(x,y) \quad \textrm{if}\quad y<0,
\end{array}\right.
\end{equation}
where 
\begin{equation}\label{eq:F+-}
F^-(x,y)=F(x,y), \quad F^+(x,y)=-RF(R(x,y)). 
\end{equation}

For $z=(x,y)^T$, we denote by $\G^{\pm}(t,z)=\big(\G^{\pm}_1(t,z),\G^{\pm}_2(t,z)\big)^T$ the solutions of systems 
$(x',y')^T=F^{\pm}(x,y)$ such that $\G^{\pm}(0,z)=z$.  Let 
\begin{equation}\label{Ycolum}
Y^{\pm}(t,z)=
D_z \G^{\pm}(t,z)=\left(\dfrac{\p \G^{\pm}}{\p x}(t,z)\quad \dfrac{\p \G^{\pm}}{\p y}(t,z)\right)
\end{equation}
be a  Fundamental Matrix Solution of the variational equations
\begin{equation}\label{vareq}
\dfrac{\p Y^{\pm}}{\p t}(t,z)=D F^{\pm}\left(\G^{\pm}(t,z)\right)Y^{\pm}(t,z),
\end{equation}
with initial condition $Y^{\pm}(0,z)=I_2$ ($2\times 2$ identity matrix). 

The following result is a straightforward consequence of the reversibility property  of the solution, $\G^+(t,z)
=R\G^-(-t,Rz)$.

\begin{lemma}\label{l1}
The equality $Y^-(t,z)=R Y^+(-t,Rz)R$ holds.
\end{lemma}

As a consequence of the above lemma, we get  
\[
D_z\G^-_1(t,z)=D_z\G_1^+(-t,Rz)R \,\,\text{and}\,\, D_z\G^-_2(t,z)=-D_z\G_2^+(-t,Rz)R.
\]

Let $F=(F_1,F_2)^T$. 
In order to assure that system \eqref{ups1} has a simple two-fold cycle (see Figure \ref{fig1}), we have to assume the following hypotheses:
\begin{itemize}
\item[$(h_1)$] 
There exist  $x_i<x_v$ such that
\begin{equation*}\label{visinv}
F_2(p_{v})=F_2(p_{i})=0, \quad \dfrac{\p F_2}{\p x }(p_v)F_1(p_v)< 0, \quad \text{and}\quad\dfrac{\p F_2}{\p x }(p_i)F_1(p_i)> 0.
\end{equation*}
where  $p_{v}=(x_{v},0) \in\Sigma$, $p_{i}=(x_{i},0) \in\Sigma$, and $F_2(x,0)\neq0$ for $x_i<x<x_v$. 

\item[$(h_2)$] 
For each $x_i<x\leq x_v$, the solution $\G^-(t,x,0)$ reaches transversely the line of discontinuity $\Sigma$ for $t=\ov\si(x)>0$, 
that is 
\begin{equation}\label{eq:sigmax}
\G^-_2(\ov\si(x),x,0)=0 \,\,\text{and}\,\,   F_2\big(\G^-(\ov\si(x),x,0)\big)\neq0.
\end{equation}

\end{itemize}

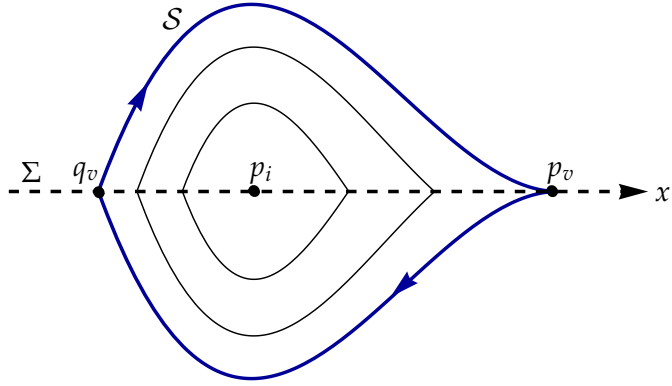
\begin{figure}[h]
\begin{center}
\begin{overpic}[width=8.5cm]{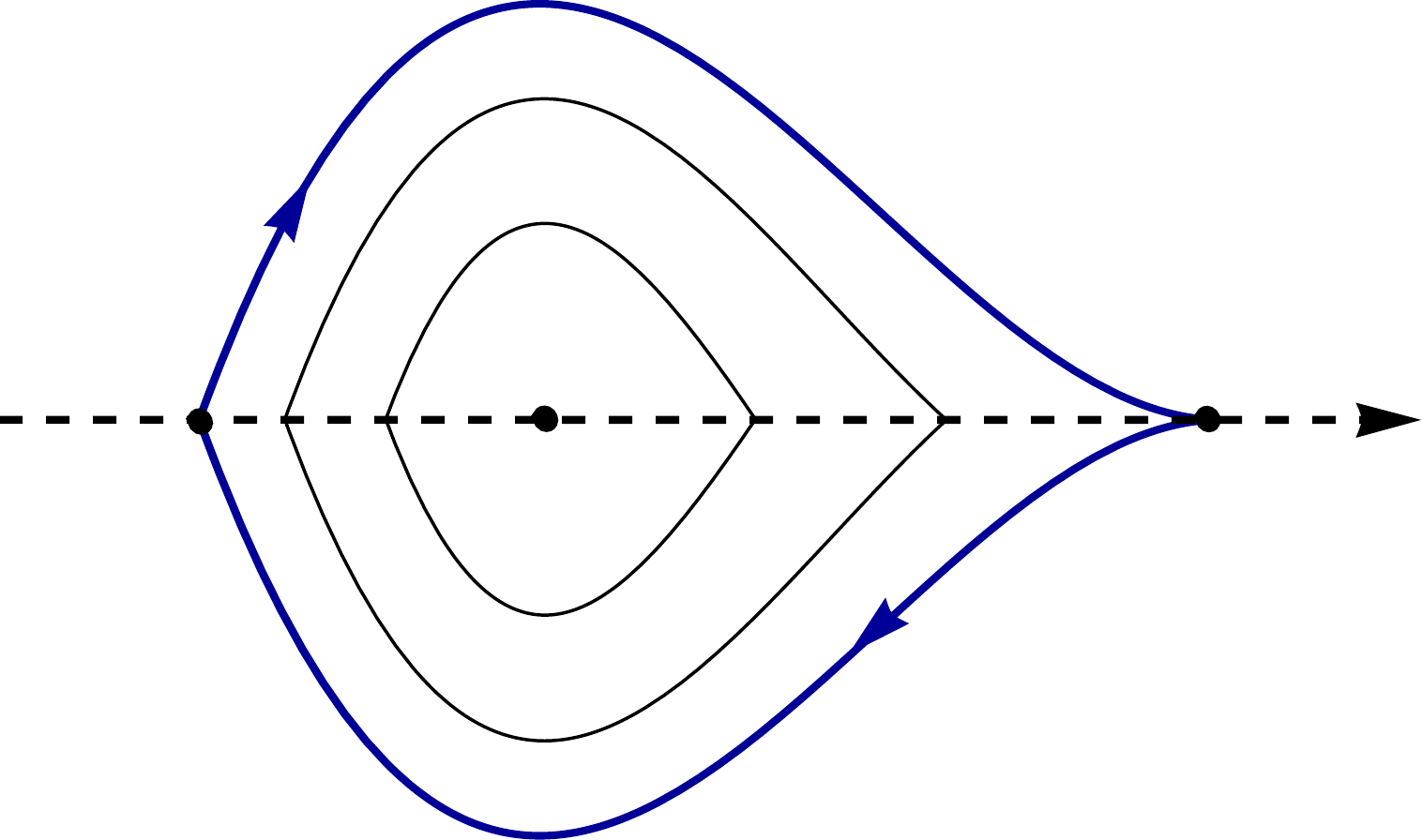}
%\begin{overpic}[grid,tics=5,width=12cm]{NPFigure.pdf}
\put(2,31){$\Sigma$} 
\put(84,32){$p_v$}
\put(37.5,32){$p_i$}
\put(10,32){$q_v$}
\put(101,28){$x$}
\put(24,55){$\CS$}
\end{overpic}
\end{center}

\caption{Periodic orbits of system \eqref{ups1} surrounding the invisible two-fold point $p_i$ and fulfilling an annulus enclosed by the 
simple two-fold cycle $\CS$.}\label{fig1}
\end{figure}

From the reversibility property of the vector field $Z_0$, hypothesis $(h_1)$ implies that  the points $p_v, p_i \in\Sigma$  are, respectively, 
visible--visible and invisible--invisible folds. 

Hypothesis $(h_2)$ fixes the orientation of the flow, which implies that
$$
F_1(p_{v,i})<0, \,\, \frac{\p F_2}{\p x}(p_v)>0,\,\,\text{and}\,\, \frac{\p F_2}{\p x}(p_i) <0.
$$  
Hypothesis $(h_1)$ also leads to the next result, which allows us to make explicit the first column of the matrix $Y^-(t,p_v)$, see \eqref{Ycolum}.
\begin{lemma}\label{l2}
For every $t\in\R,$ the following equality holds
\[
\dfrac{\p \Gamma^{-}}{\p x}(t,p_v)=\dfrac{F\big(\Gamma^{-}(t,p_v)\big)}{F_1(p_v)}.
\]
\end{lemma}
\begin{proof}
First we note that, as $F^-=F$,  the function $w(t)=\dfrac{\p \Gamma^{-}}{\p x}(t,p_v)$ is a solution of the differential equation 
$\dot w=D_zF\big(\G^{-}(t,p_v)\big)w$ with the initial condition $w(0)=(1,0)$. 
Now, take 
\[
\ov w(t)=\dfrac{\dfrac{\p \Gamma^{-}}{\p t}(t,p_v)}{F_1(p_v)}=\dfrac{F\big(\Gamma^{-}(t,p_v)\big)}{F_1(p_v)}.
\]
Computing its derivative with respect to the variable $t$ we have
\[
\begin{array}{RL}
\dfrac{d\ov w}{dt}(t)=D_zF\big(\Gamma^{-}(t,p_v)\big)\dfrac{\dfrac{\p \Gamma^{-}}{\p t}(t,p_v)}{F_1(p_v)}=D_zF\big(\Gamma^{-}(t,p_v)\big)\ov w(t).
\end{array}
\]
Moreover, hypothesis $(h_1)$ implies that
\[
\ov w(0)=\dfrac{F(p_v)}{F_1(p_v)}=\left(\dfrac{F_1(p_v)}{F_1(p_v)},\dfrac{F_2(p_v)}{F_1(p_v)}\right)=(1,0)=w(0).
\]
Hence, we conclude that $w(t)=\ov w(t)$.
\end{proof}

Hypothesis $(h_2),$ together with the reversibility property, imply that for each $x_i<x\leq x_v$ the function
\begin{equation}\label{gamma}
\g(t,x)=\big(\g_1(t,x),\g_2(t,x)\big)=\left\{
\begin{array}{ll}
\G^-(t,x,0)&\textrm{if}\quad 0\leq t \leq \ov\si(x),\vspace{0.2cm}\\
R\G^-(-t,x,0)&\textrm{if}\quad -\ov\si(x)\leq t \leq 0
\end{array}\right.
\end{equation}
is a $2\ov\si(x)$--periodic solution of system \eqref{ups1} such that $\g(0,x)=(x,0)\in\Sigma$.
Consequently, the invisible two fold $p_i$ behaves as a center having an annulus of periodic orbits ending at the simple two-fold cycle 
$\CS=\{\g(t,x_v):\,-\ov\si(x_v)\leq t\leq \ov\si(x_v)\}$ (see Figure \ref{fig1}). 
Notice that
\begin{equation}\label{eq:qv}
\CS\cap\Sigma=\{p_v,q_v\},\quad \mbox{where}\quad q_v=\Gamma^-(\ov\si(x_v),p_v).
\end{equation}
From now on, when it is convenient, we shall denote $\G^-$ and $Y^-$, only by $\G$ and $Y$, respectively.  

We note that, by hypothesis $(h_2)$, the function $\ov\si(x)$ is differentiable on the interval  
$(x_i,x_v].$ Indeed, it is a solution of the implicit equation $\G_2(\ov{\si}(x),x,0)=0$. 
Differentiating this last equality implicitly in the variable $x$ we obtain,
for each $x_i<x\leq x_v$, the following relation
\begin{equation}\label{eq:sigmaprima}
\ov\si'(x)=-\dfrac{\dfrac{\p \G_2}{\p x}(\ov\si(x),x,0)}{F_2(\G(\ov\si(x),x,0))}.
\end{equation}
Furthermore, since $p_i=(x_i,0)$ is an invisible-invisible fold, then $\inf\{\ov\si(x):\,x_i<x\leq x_v\}=0,$ and 
$0\leq \si_M=\sup\{\ov\si(x):\,x_i<x\leq x_v\}<\infty$. 
Accordingly, we fix the interval $\CT=[0,\si_M]$.

\section{Statement of the main results}\label{MR}

We consider the following perturbation of system \eqref{ups1}. 
\begin{equation}\label{ps1}
(x',y')^T=Z_\e(t,x,y)=\left\{\begin{array}{l} X^+_{\e}(t,x,y) \quad \textrm{if} \quad y>0,\vspace{0.2cm}\\
X^-_{\e}(t,x,y) \quad \textrm{if}\quad y<0,
\end{array}\right.
\end{equation}
where
\[
X^{\pm}_{\e}(t,x,y)=F^{\pm}(x,y)+\e \,G^{\pm}(t,x,y)+\e^2 H^{\pm}(t,x,y;\e).
\]
We assume that $G^{\pm}(t,x,y)$ and $H^{\pm}(t,x,y;\e)$ are smooth functions in $\R\times D$ and $\R\times D\times (-\e_0,\e_0)$, respectively,  
and $2\si$--periodic in the variable $t$  for some $\si\in\CT=[0,\si_M]$.

We want to detect, for $\e>0$ small enough, the existence of isolated $2\sigma$-periodic solutions of system \eqref{ps1}.  First, we notice that if 
\begin{equation}\label{revcond}
X^-_{\e}(t,z)+R X^+_{\e}(s,R z)\equiv 0, \text{ for } (t,s,z)\in \R^2\times D,
\end{equation}
then,  for $|\e|\neq0$ sufficiently small, every periodic solution of $Z_0(x,y)$ persists for $Z_{\e}(t,x,y).$  Indeed, \eqref{revcond} implies that $Z_{\e}(t,x,y)$ is autonomous and $R$-reversible. Taking \eqref{revcond} into account, we define the following operator:
\begin{equation}\label{rever}
\big\{X^+,X^-\big\}_{\T}(t,z)=X^-(t+\T,z)+RX^+(-t+\T,Rz).
\end{equation}
Notice that $\big\{X^+_{\e},X^-_{\e}\big\}_{\T}$ can be seen as a measurement of the non-reversibility of $Z_{\e}.$ Indeed, $\big\{X^+_{\e},X^-_{\e}\big\}_{\T}\equiv0$ is equivalent to condition \eqref{revcond}. Thus, $\big\{X^+_{\e},X^-_{\e}\big\}_{\T}\not\equiv0$ is a necessary condition for the existence of isolated periodic solutions of $Z_{
\e}$. Computing the expansion of $\big\{X^+_{\e},X^-_{\e}\big\}_{\T}$ around $\e=0,$ we see that $\big\{G^-,G^+\big\}_{\T}\neq0$ implies $\big\{X^+_{\e},X^-_{\e}\big\}_{\T}\neq 0,$ for $|\e|\neq0$ sufficiently small. The value $\big\{G^-,G^+\big\}_{\T}$ will be important for the definition of the Melnikov function.

Accordingly, let $\s^1_{\si}\equiv \R/(2\si\Z)$ and define the Melnikov function $M:\s^1_{\si}\times(x_i,x_v]\rightarrow \R$ as
\begin{equation}\label{mel}
M(\T,x)=F(\g(\ov{\si}(x),x))\wedge \left(Y(\ov{\si}(x),x,0)\!\!\!\int_0^{\bar{\si}(x)}\!\!\!\!\!\!\!\! Y(t,x,0)^{-1}\big\{G^-,G^+\big\}_{\T}(t,\g(t,x))dt\right).
\end{equation}
where $\gamma$ is given in \eqref{gamma}, $Y$ is the fundamental matrix given in \eqref{Ycolum}.
Here, the wedge product is defined by $(a_1,a_2)\wedge(b_1,b_2)=\langle (-a_2,a_1),(b_1,b_2)\rangle$. As mentioned before, the expression \eqref{mel} will be obtained through standard analysis of the expansion of a suitable displacement function around $\e=0.$ A similar Melnikov function was obtained in \cite{GouveiaEtAl16} for autonomous perturbations of a $n$-dimensional non-smooth system with a codimension-1 period annulus.

\subsection{Bifurcations from the period annulus $\CA$}
Our first main result is concerned about the existence of isolated crossing periodic solutions of system \eqref{ps1} bifurcating from the period annulus $\CA.$

This kind of problem has been studied in a rather general setting for smooth systems (see, for instance, \cite{CanLliNov17,LliNovTei14} and the references therein). When dealing with non-smooth systems, the geometry of the discontinuity manifold has a strong influence on the bifurcation functions controlling the existence of isolated crossing periodic solutions. Due to this fact, the existing results in the research literature usually assume some constraints either on the unperturbed system, on the discontinuity manifold, or on the perturbation. For vanishing unperturbed systems, the averaging theory \cite{Sanders07} provides the first order bifurcation function (see \cite{LliMerNov15,LliNovTei15a}). The bifurcation functions at any order were obtained in \cite{LliNovRod17} when the discontinuity appears only in the time variable, and up to order $2$ in \cite{BastosEtAl19} for more general discontinuity manifold. For non-vanishing unperturbed systems with a period annulus of crossing periodic orbits, the bifurcation functions at any order were obtained in \cite{LliNovRod2020} assuming again that the discontinuity appears only in the time variable. In \cite{GHS},  the Melnikov function was obtained for non-autonomous perturbation of a class of planar piecewise Hamiltonian systems, and in \cite{Han15} for autonomous perturbations of general planar piecewise Hamiltonian systems. In \cite{GouveiaEtAl16}, the Melnikov function was obtained for autonomous perturbation of non-smooth period annulus in $\R^n.$ In \cite{BattelliFeckan08}, a Melnikov function was obtained for studying the persistence of homoclinic trajectories in non-smooth systems. None of the mentioned results can be directly applied in our case.

\begin{mtheorem}\label{ta}
Take $\si\in\CT=[0,\si_M]$ and $x_\si\in (x_i,x_v)$ such that $\ov\si(x_\si)=\si$  and $\ov\si'(x_\si)\neq0$, where $\ov \si(x)$ is given in \eqref{eq:sigmax}.
Assume that the vector field $Z_\e$ in \eqref{ps1} is $2\si$--periodic in the variable $t$.
If there exists $\T^*\in\s^1_{\si}$ such that
$$
M(\T^*,x_\si)=0 \quad \mbox{and} \quad \frac{\p M}{\p\T} (\T^*,x_\si)\neq 0,
$$ 
then for $\e>0$ sufficiently small there exists an isolated crossing $2\si$--periodic solution of system \eqref{ps1} with initial condition, 
in $\s^1_{\si}\times D$,  $\e$-close to $(t_0,z_0)=(\T^*,(x_{\si},0))$.
\end{mtheorem}

The next result is obtained as a consequence of Theorem \ref{ta} and deals with the continuation problem of subharmonic 
crossing periodic solutions of system \eqref{ps1} when it is autonomous.

\begin{corollary}\label{ca}
Assume that the vector field $Z_\e$ in \eqref{ps1} is autonomous and denote 
$M(x)=M(\T,x)$. 
If there exists $x^*\in(x_i,x_v)$ such that $M(x^*)=0$ and $M'(x^*)\neq0$ then, for $\e>0$ sufficiently small, 
there exists a crossing periodic solution of system \eqref{ps1} with initial condition, in $D,$ $\e$-close to $(x^*,0)$.
\end{corollary}

\subsection{Bifurcations from the two-fold connection $\CS$}
Our second main result is concerned about the bifurcation of periodic solutions  from the simple two-fold connection $\CS$ in the special case that system \eqref{ps1} is perturbed by 
$2\si_v=2\ov\si(x_v)$--periodic functions. 
 This problem resembles the bifurcation of periodic solutions from saddle homoclinic connections in smooth systems. Indeed, $\CS$ is a boundary of a period annulus $\CA,$ with the difference that a trajectory connects the two-fold singularity to itself in a finite time, namely $2\sigma_v$.
We shall see that, in this case, the unfolding of $\CS$ gives rise to sliding dynamics and either a crossing or a sliding periodic solution can appear. Therefore, the standard analysis performed in Theorem \ref{ta} does not apply here.

For each $\T\in\s^1_{\si_v}$, we define the number $g_{\T}\in\R$  as
\begin{equation}\label{g}
\begin{array}{rcll}
g_\T &=&\Big\langle D_z \Gamma_2(\si_v,p_v)\,,\,\displaystyle  \int_0^{\si_v}
\!\!\!\!\!\! Y(t,p_v)^{-1}\big\{&\!\!\!\!\!G^-,-G^+\big\}_{\T}(t,\G(t,p_v))dt\Big\rangle \vspace{0.2cm} \\
&=&\Big\langle D_z \Gamma_2(\si_v,p_v)\,,\, \displaystyle \int_0^{\si_v}
\!\!\!\!\!\! Y(t,p_v)^{-1}\Big(&\!\!\!\!\!G^-(t+\T,\G(t,p_v))\\
&&&\!\!\!\!\!-RG^+(-t+\T,R\G(t,p_v))\Big)dt\Big\rangle.
\end{array}
\end{equation}

In the above expression, the inner product notation $\langle*,*\rangle$ is actually an abuse of notation. Indeed, the left and right factors are expressed as row and column vectors , respectively.  Thus, the matrix product between them results in a scalar. Nevertheless, due to the amount of computations involving matrices, we decide to consider the inner product notation to emphasize that the result is in fact a scalar, avoiding then any possible misunderstanding. 

\begin{mtheorem}\label{tb}
Suppose that the vector field $Z_\e$ in \eqref{ps1} is $2\si_v$--periodic in the variable $t$
and assume
that there exists $\T^*\in\s^1_{\si_v}$ such that $M(\T^*,x_v)=0$ and $(\p M/\p\T)(\T^*,$ $x_v)\neq0$.
\begin{itemize}
\item[{\bf (a)}]
If $G^+_2(\T^*,p_v)\neq G^-_2(\T^*,p_v)$ and 
\[
g_{\T^*}>\dfrac{2F_2(q_v)}{F_1(p_v)\dfrac{\p F_2}{\p x}(p_v)}\max\big\{G^{\pm}_2(\T^*,p_v)\big\},
\]
then, for $\e>0$ sufficiently small, there exists a sliding $2\si_v$--periodic solution of system \eqref{ps1} with initial condition, 
in $\s^1_{\si_v}\times D$, $\e$-close to $(t_0,z_0)=(\T^*,p_v)$.
Moreover, this solution slides either on $\Sigma^s$ or $\Sigma^e$ provided that $G^+_2(\T^*,p_v)< G^-_2(\T^*,p_v)$  or 
$G^+_2(\T^*,p_v)> G^-_2(\T^*,p_v)$, respectively.

\item[{\bf (b)}] If 
\[
g_{\T^*}<\dfrac{2F_2(q_v)}{F_1(p_v)\dfrac{\p F_2}{\p x}(p_v)}\max\big\{G^{\pm}_2(\T^*,p_v)\big\},
\]
then, for $\e>0$ sufficiently small, there exists a crossing $2\si_v$--periodic solution of system \eqref{ps1} with initial condition, 
in $\s^1_{\si_v}\times D$,  $\e$-close to $(t_0,z_0)=(\T^*,p_v)$.
\end{itemize}
\end{mtheorem}

\section{A piecewise Hamiltonian model}\label{exam}
In this section, we apply the previous results to study the crossing and sliding periodic solutions of non-autonomous 
perturbations of a piecewise Hamiltonian model.  
This kind of problem had been previously addressed in \cite{KKY}, where the authors applied KAM theory to prove that, under certain conditions, a piecewise Hamiltonian model has  infinitely many periodic solutions.

Consider the following continuous Hamiltonian function
\[
H(x,y)=|y|-\dfrac{x^3}{3}+\al x, \quad \text{where} \quad \al>0.
\]
As usual, $|\cdot|$ denotes the absolute value of a real number. 
The above Hamiltonian gives rise to the following discontinuous piecewise Hamiltonian differential system
\begin{equation}\label{ZH0}
(x',y')^T=Z_0^\al(x,y)=(\sgn(y),x^2-\al)  =\left\{\begin{array}{l}  (1,x^2-\al)  \quad \textrm{if} \quad y>0,\vspace{0.2cm}\\
(-1,x^2-\al)    \quad \textrm{if}\quad y<0,
\end{array}\right.
\end{equation} 
The switching surface is given by $\Sigma=\{(x,0),\,x\in\R\}$. 
Its phase space is depicted in Figure \ref{piecewise}. 
Following the notation of the previous Section we take 
\[
F^-(x,y)=(-1,x^2-\al), \quad  F^+(x,y)= (1,x^2-\al). 
\]
Notice that the above  piecewise Hamiltonian differential system $Z_0^\al$ is $R$--reversible with $R(x,y)=(x,-y)$. 
In addition, it has two two-fold singularities, one invisible $p_i=(x_i,0)=(-\sqrt{\al},0)$ and other visible $p_v=(x_v,0)=(\sqrt{\al},0)$. 

The solution $\G^-(t,x,y)$ of  $(x',y')^T =F^-(x,y)$  can be easily computed as
\[
\Gamma^-(t,x,y)=\left(-t + x, \displaystyle\frac{1}{3} (t^3 - 3 t^2 x + 3 t x^2 +3 y- 3 t \alpha)\right).
\]
Furthermore, for each $-\sqrt{\al} < x \leq  \sqrt{\al}$, it is straightforward to see that $\Gamma^-(t,x,0)$ reaches transversely 
$\Sigma$ for 
$t=\ov \sigma(x)=\displaystyle \frac{1}{2} (3 x + \sqrt{3} \sqrt{-x^2 + 4 \alpha})$.
Hence, for $-\sqrt{\al} < x \leq \sqrt{\al}$,  the reversibility property implies that the solution $\gamma(t,x)$ of \eqref{ZH0}, 
satisfying $\gamma(0,x)=(x,0)$, is given by
\[
\gamma(t,x)=\left\{
\begin{array}{ll}
\left(-t + x, \displaystyle\frac{1}{3} (t^3 - 3 t^2 x + 3 t x^2 - 3 t \alpha) \right)&\textrm{if}\quad 0\leq t \leq \ov{\sigma}(x),\vspace{0.2cm}\\
\left(t + x, \displaystyle\frac{1}{3} (t^3 + 3 t^2 x + 3 t x^2 - 3 t \alpha)\right)&\textrm{if}\quad -\ov{\sigma}(x)\leq t \leq 0\vspace{0.2cm},\\
\end{array}\right.
\]
From the formula of $\ov \sigma (x)$, one obtains an explicit expression for the point $(x_\sigma,0)$ satisfying $\ov \sigma(x_{\sigma})=\sigma$,
\begin{equation}\label{eq:xsigma}
x_\si = \displaystyle \frac{1}{6} (3 \si - \sqrt{3} \sqrt{12 \al - \si^2})\in ( -\sqrt{\al},\sqrt{\al}). 
\end{equation}
Accordingly, $Z_0^\al$ satisfies hypotheses $(h_1)$ and $(h_2)$. 
Furthermore, since $\ov \sigma(\sqrt{\alpha}) =3\sqrt{\alpha}$, we get 
$\CS = \{ \gamma(t,\sqrt{\alpha}): -3\sqrt{\alpha}\leq t \leq 3\sqrt{\alpha}  \}$ 
(see Figure \ref{piecewise}).  Clearly $\CS \cap \Sigma =\{p_v,q_v\},$ with $q_v =\Gamma^-(\ov{\sigma}(\sqrt{\alpha}),\sqrt{\alpha},0)=(-2\sqrt{\alpha},0)$.

\subsection{non-autonomous perturbation}

Now, in order to illustrate the application of Theorems \ref{ta} and \ref{tb}, we consider the following non-autonomous perturbation of \eqref{ZH0}.
\begin{equation}\label{eq1}
(x',y')^T=Z_\e(x,y)=\left\{\begin{array}{l} F^+(x,y) +\e G^+(t,x,y)\quad \textrm{if} \quad y>0,\vspace{0.2cm}\\
F^-(x,y) +\e G^-(t,x,y) \quad \textrm{if}\quad y<0,
\end{array}\right. 
\end{equation}
where
\[
G^+(t,x,y)=\left(0,\la \sin  \frac{\pi t}{\si} \right)\quad \text{and} \quad G^-(t,x,y) =\left(0,\sin  \frac{\pi t}{\si}\right),
\] 
for some $\lambda \in \R$. Notice that $G^{\pm}(t,x,y)$ are $2\si$--periodic in the variable $t$. 
We shall see that, for convenient values of $\lambda$, system \eqref{eq1} satisfies the hypotheses either of Theorem \ref{ta} or Theorem \ref{tb}. 

The fundamental matrix solution $Y(t,x,y)=Y^-(t,x,y)$, defined in \eqref{Ycolum}, is given by
\[
Y(t,x,y) = \left( \begin{matrix}
1 & 0 \\
- t^2 + 2 t x &1
\end{matrix} \right).
\]
Thus, we compute the function \eqref{mel} as
\begin{align*}
M(\T,x) =&\displaystyle \frac{\si}{\pi} \bigg(   \la \cos \Big[\frac{\pi (3 x + \sqrt{3} \sqrt{ 4 \alpha-x^2} + 2 \T)}{2 \si}\Big]  +\cos \Big[\frac{\pi (3 x + \sqrt{3} \sqrt{ 4 \alpha-x^2} - 2 \T)}{2 \si}\Big] \\ 
& -(1+\la)\cos\Big[ \frac{\pi \T}{\si }\Big]  \bigg).
\end{align*}

In the next result, as an application of Theorem \ref{ta}, we show that system \eqref{eq1} has  two crossing periodic solutions, 
provided that the period of the perturbation is strictly less than $6\sqrt{\al}$.
\begin{proposition}\label{prop:exA}
Assume that $\la\neq-1$. 
Then, for each $\sigma\in(0,3\sqrt{\al})$ and for $\e > 0$ sufficiently small,  there exist two crossing $2\sigma$--periodic solutions of system 
\eqref{eq1} with initial conditions $\e$-close to $(3\si/2,(x_{\si},0))$ and $(\si/2,(x_{\si},0))$, respectively (see Figure \ref{fig-exampleA}). 
\end{proposition}

\begin{proof}
Given $\si \in (0,3 \sqrt{\al})$, notice that  $\ov\si(x_\si) =\sigma$  if, and only if, 
$x_\si = \displaystyle \frac{1}{6} (3 \si - \sqrt{3} \sqrt{12 \al - \si^2})\in ( -\sqrt{\al},\sqrt{\al})$ (see \eqref{eq:xsigma}). 
Then,
\[
\begin{array}{ll}
M(\T,x_\si) =& -\displaystyle\frac{2(1+\la)\si}{\pi}  \cos\left( \frac{\pi \T}{\si} \right),
\end{array}
\]
where we used the following relation
\[
\sqrt{36\al+2\si\left(\sqrt{36\al-3\si^2}-\si\right)}=\si+\sqrt{36\al-3\si^2},
\]
for every $\al>0$ and $\si\in[0,2\sqrt{3\al}]$.

Solving $M(\T,x_\si)=0$, for $\T\in\s^1_{\si}$, we get $\T^*_1=3\si/2$ and $\T^*_2=\si/2$.
Moreover,
\[
\dfrac{\p M}{ \p \T}(\T^*_2,x_\si)=-\dfrac{\p M}{ \p \T}(\T^*_1,x_\si)=2(1+\la)\neq0.
\]
Hence, the proof follows from Theorem \ref{ta}.
\end{proof}

\begin{figure}[H]
\begin{center}
\begin{overpic}[width=12cm]{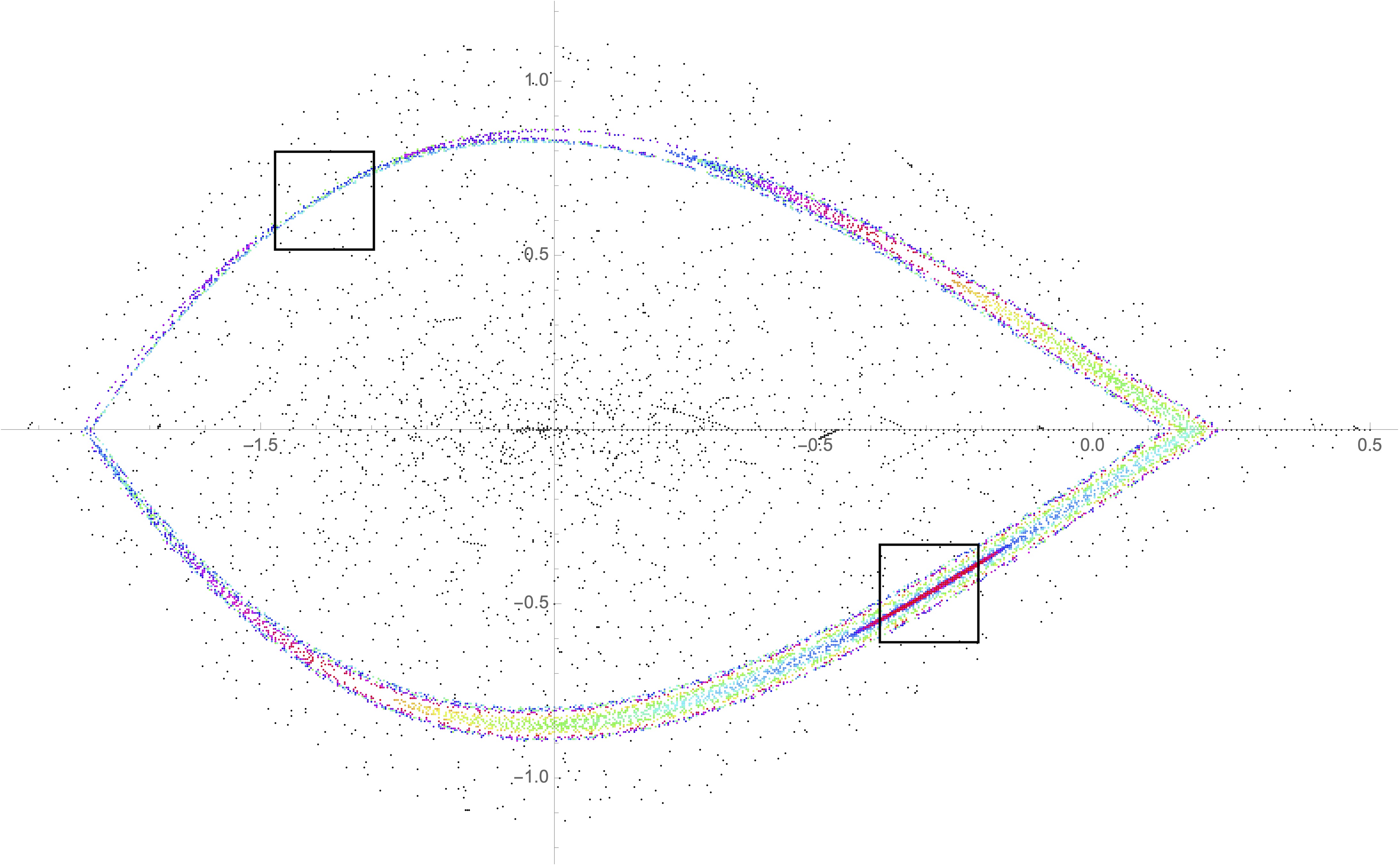}
%\begin{overpic}[grid,tics=5,width=14cm]{CrossingPoincareMap-small.png}
\put(101,31){$x$} 
\put(39,63){$y$} 
\end{overpic}
\end{center}

\bigskip

\caption{Numerical simulation of the Poincar\'{e} map of system \eqref{eq1} assuming $\alpha=1,$ $\lambda=2,$ $\sigma=2,$ and $\e=1/1500.$ The boxes spot the fixed points corresponding to crossing periodic solutions predicted by Proposition \ref{prop:exA}.}\label{fig-exampleA}
\end{figure}

In the next result, as an application of Theorem \ref{tb}, we were able to detect crossing and sliding periodic solutions of system \eqref{eq1} 
provided that the period of the perturbation is equal to $6\sqrt{\al}$.
\begin{proposition} \label{prop2} Assume that $\si=3\sqrt{\al}$ and $\la\neq-1.$ Then, for $\e > 0$ sufficiently small, the following statement holds
\begin{itemize}
\item [$(i)$] 
for $\la\neq 1$, there exists a sliding $6 \sqrt{\al}$--periodic solution of system \eqref{eq1}  with initial condition $\e$-close to 
$(3\sqrt{\alpha}/2,(x_{v},0))$, which slides  either on $\Sigma^s$ or $\Sigma^e$  provided that $\la<1$  or $\la>1$ (see Figure \ref{fig-exampleB});
\item [$(ii)$] 
for $\la<0$, there exists a sliding $6 \sqrt{\al}$--periodic solution of system \eqref{eq1} with initial conditions $\e$-close to 
$( 9\sqrt{\alpha}/2,(x_{v},0))$, which slides on  $\Sigma^e$;
\item [$(iii)$] 
for $\la>0$, there exists a crossing $6 \sqrt{\al}$--periodic solution of system \eqref{eq1} with initial conditions $\e$-close to $( 9\sqrt{\alpha}/2,(x_{v},0))$.
\end{itemize} 
\end{proposition}

\begin{remark}
Notice that, from Proposition \ref{prop2}, sliding and crossing periodic solutions may coexist. 
More specifically, comparing the statements $(i)$ and $(ii)$ we get 
\begin{itemize}
\item 
For $\lambda\in(-\infty,0)\setminus\{-1\},$ there exists two sliding $6 \sqrt{\al}$--periodic solutions. 
One with initial condition $\e$-close to $(3\sqrt{\alpha}/2,(x_{v},0))$, which slides on $\Sigma^s$, and another with initial conditions $\e$-close to 
$( 9\sqrt{\alpha}/2,(x_{\si},0))$, which slides on  $\Sigma^e$.
\item 
For $\lambda\in(0,1),$ there exist a crossing $6 \sqrt{\al}$--periodic solution with initial conditions $\e$-close to 
$( 9\sqrt{\alpha}/2,(x_{\si},0))$  and a sliding $6 \sqrt{\al}$--periodic solution  with initial condition $\e$-close to 
$(3\sqrt{\alpha}/2,(x_{v},0))$, which slides  on $\Sigma^s$.
\item 
For $\lambda\in(1,+\infty),$ there exist a crossing $6 \sqrt{\al}$--periodic solution with initial conditions $\e$-close to 
$( 9\sqrt{\alpha}/2,(x_{\si},0))$  and a sliding $6 \sqrt{\al}$--periodic solution  with initial condition $\e$-close to 
$(3\sqrt{\alpha}/2,(x_{v},0))$, which slides  on $\Sigma^e$.
\end{itemize}
\end{remark}

\begin{proof}[ Proof of Proposition \ref{prop2}]
If  $\si = 3\sqrt{\alpha}$, then $G^{\pm}(t,x,y)$ are $6 \sqrt{\al}$--periodic in the variable $t$, and 
\[
M(\T,\sqrt{\alpha}) =   \displaystyle \frac{-6 \sqrt{\alpha} (1+\la)}{\pi}  \cos\left(\frac{\pi \T} {3 \sqrt{\alpha}}\right).
\]
Solving $M(\T^*,\sqrt{\alpha})=0$, for $\T\in[0,6\sqrt{a}]$, we get $\T^*_1 = 3\sqrt{\alpha}/2 $ and $\T^*_2 = 9\sqrt{\alpha}/2$. Moreover,
\[
\displaystyle \frac{\p M}{ \p \T}(\T^*_2,\sqrt{\alpha})=-\displaystyle \frac{\p M}{ \p \T}(\T^*_1,\sqrt{\alpha})=  2(1+\la)\neq0,
\]
and  $g_{\T}=\displaystyle \frac{6 \sqrt{\alpha}(1 -\la)}{\pi}  \cos  \left( \frac{\pi \T} {3 \sqrt{\alpha}} \right).$ Thus, $g_{\T^*_{1,2}}=0$. Furthermore,
$G^{+}_2(\T^*_n,p_v)= (-1)^{(1 + n)} \la $, $G^{-}_2(\T^*_n,p_v) =   (-1)^{(1 + n)}$, for $n=1,2$, and  
 \[
\dfrac{2F_2(q_v)}{F_1(p_v)\dfrac{\p F_2}{\p x}(p_v)}=  -3 \sqrt{\alpha}.
\]

To obtain statement $(i)$ notice that, for $\la\neq1$, $ G_2^+(\T_1^*,p_v) \neq  G_2^-(\T_1^*,p_v)$. In this case, 
\[ 
g_{\T_1^*} =0> -3 \sqrt{\alpha}  \max\big\{G^{\pm}_2(\T^*_1,p_v)\big\}=-3 \sqrt{\alpha}  \max\big\{1,\la\big\}.
\]
Therefore, from statement $(a)$ of Theorem \ref{tb}, there exists a sliding $6\sqrt{\al}$--periodic solution with initial condition $\e$-close to $(3\sqrt{\al}/2,p_v)$. 
Moreover, for $\la>1$, we have $ G_2^+(\T_1^*,p_v)> G_2^-(\T_1^*,p_v),$ which implies that this periodic solution slides on $\Sigma^e$. 
Analogously, for $\la<1$, we have $G_2^+(\T_1^*,p_v)< G_2^-(\T_1^*,p_v),$ which implies that this periodic solution slides on $\Sigma^s$.

To obtain statement $(ii)$ notice that,  for $\la <0$, we have
\[
g_{\T_2^*}=0 > -3 \sqrt{\alpha}  \max\big\{G^{\pm}_2(\T^*_2,p_v)\}=-3 \sqrt{\alpha} \max\big\{-1,-\la\}=3 \sqrt{\alpha} \la.
\]
Therefore, from statement $(a)$ of Theorem \ref{tb}, there exists a sliding $6\sqrt{\al}$--periodic solution with initial condition $\e$-close to $(9\sqrt{\al}/2,p_v)$
Moreover,  in this case, $G_2^+(\T_2^*,p_v)>G_2^-(\T_2^*,p_v)$, which implies that this periodic solution slides on $\Sigma^e$.

Finally, to obtain statement $(iii)$ notice that, for $\la >0$, we have 
\[ 
g_{\T_2^*} =0< -3 \sqrt{\alpha}  \max\big\{G^{\pm}_2(\T^*_2,p_v)\}=-3 \sqrt{\alpha} \max\big\{-1,-\la\}.
\]
Therefore, from statement $(b)$ of Theorem \ref{tb}, there exists a crossing $6\sqrt{\al}$--periodic solution with initial condition $\e$-close to $(9\sqrt{\al}/2,p_v)$.
\end{proof}

\begin{figure}[H]
\begin{minipage}{8cm}
\centering
\begin{overpic}[width=8.5cm]{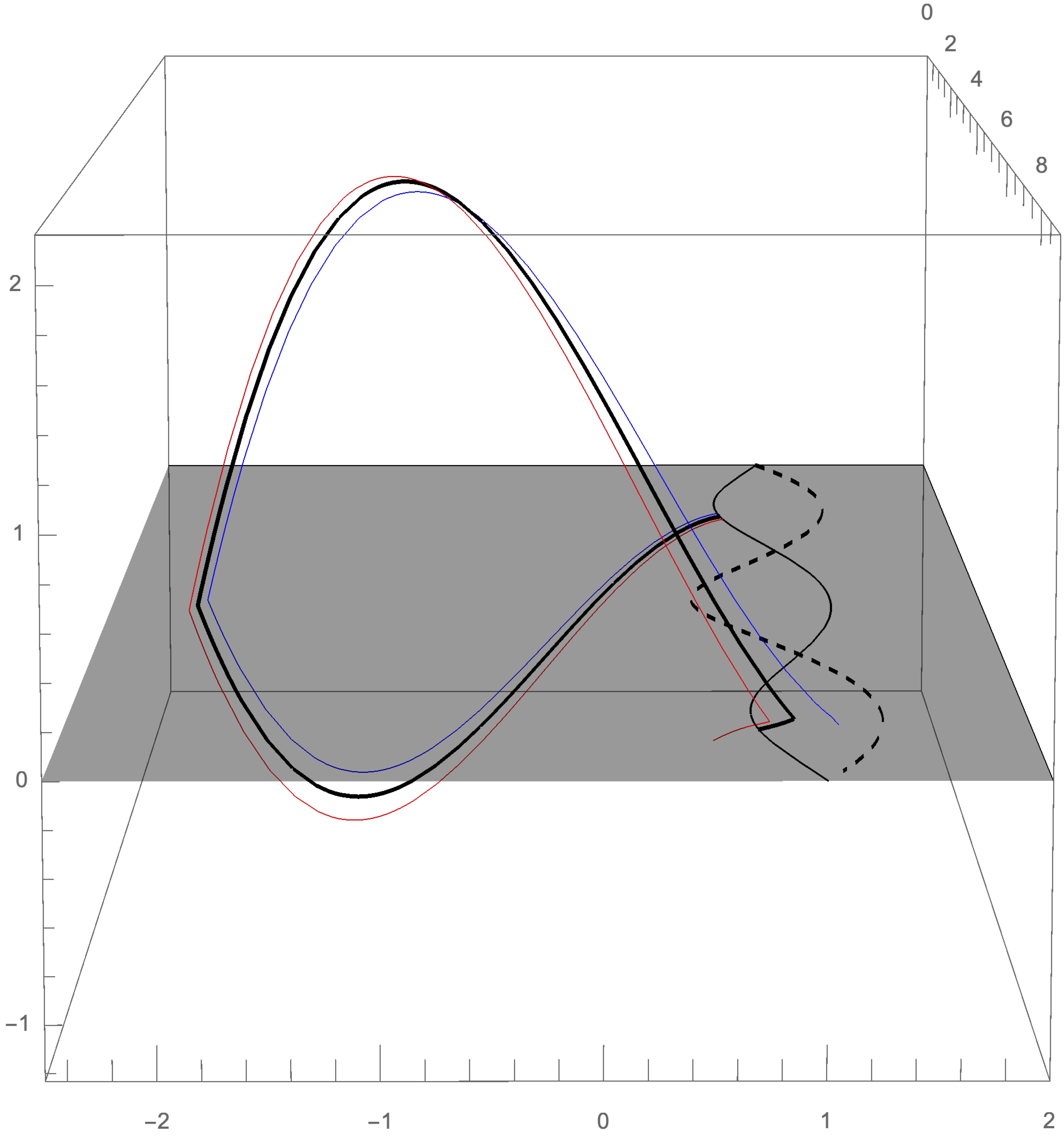}
%\begin{overpic}[grid,tics=5,width=8.5cm]{Sliding3D.png}
\linethickness{0.75pt}
\put(68.5,35.5){\vector(1,-0.5){42}}
\put(148,43){\vector(-0.3,-1){8.5}}
\put(110,10){sliding segment}
\put(7,33){$\Sigma$} 
\put(18,8){$x$} 
\put(0,20){$y$} 
\put(10,20){$t$} 
\end{overpic}
\end{minipage}
\begin{minipage}{6.4cm}
\centering
\begin{overpic}[width=5.5cm]{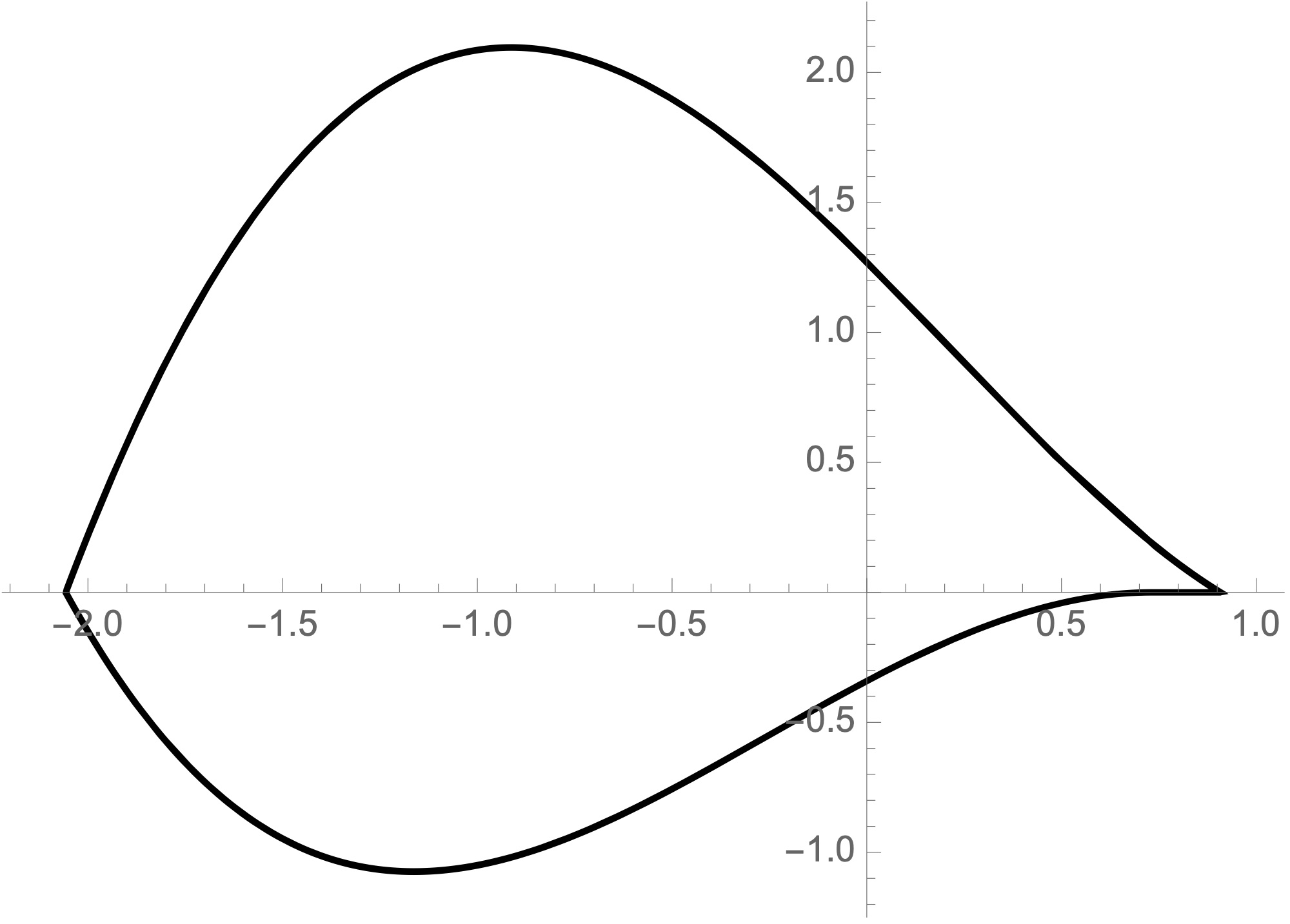}
%\begin{overpic}[grid,tics=5,width=5.5cm]{SlidingXY.pdf}
\put(102,26){$x$} 
\put(65,75){$y$} 
\end{overpic}
\end{minipage}

\caption{Numerical simulation of a sliding periodic solution predicted by Proposition \ref{prop2} (i) for system \eqref{eq1} assuming $\alpha=1,$ $\lambda=-3/2,$ $\sigma=3\sqrt{\al}=3,$ and $\e=1/2.$ The bold trajectory starts at the $2\sigma$-periodic visible fold curve of $X^-_{\e},$ with initial time condition $t_0$ near to $3\sqrt{\al}/2,$ then it crosses the discontinuity manifold, reaches the sliding region, and slides on it reaching again the visible fold curve of $X^-_{\e}$ at a time $t_0+2\sigma.$}\label{fig-exampleB}
\end{figure}
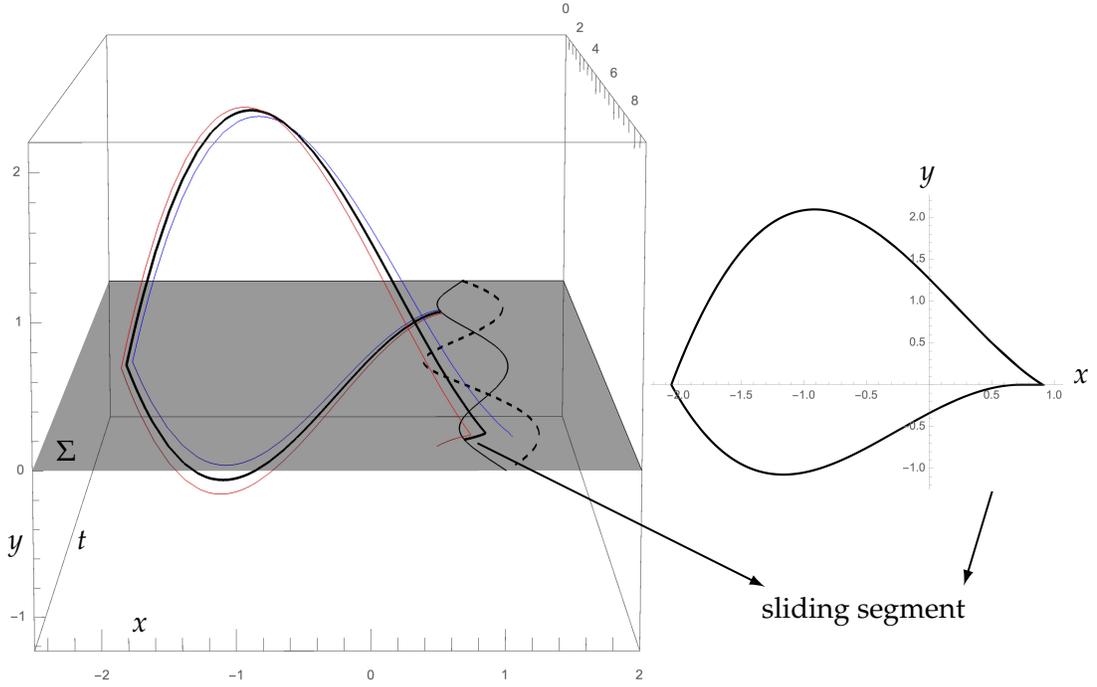

\section{Proofs of Theorems \ref{ta} and \ref{tb}}\label{proof}

Recall that to study a  non-autonomous periodic differential equation 
$w'=f(t,w)$, $(t, z)\in \s^1\times D$ 
we can work in the extended phase space adding time as a variable   
$\T'=1$ and $ v'=f(\T,v)$. If $(\T(t),v(t))$ is a solution of the autonomous system such that $(\T(0),v(0))=(\ov{\T},\ov v),$ 
then $v'(t)=f(\ov \T+t,v(t))$ and $w(t):=v(t-\ov \T)$ is the solution of the non-autonomus system such that $w(\ov \T)=\ov v$.

Accordingly, we study system \eqref{ps1} in the extended phase space
\begin{equation}\label{s1b}
\T'=1, \quad (x',y')^T=Z_\e(\T,x,y),
\end{equation}
where $(\T, x,y)\in \s_{\si}^1\times D$, $D\subset\R^2$, 
being $\s^1_{\si}\equiv \R/(2\si\Z)$. 
We note that \eqref{s1b} is also a Filippov system having $\widetilde{\Sigma}=\s^1_{\si}\times \Sigma$ as its discontinuity manifold. 
Moreover, $\widetilde{\Sigma}=\tilde h^{-1}(0)$ for $\tilde h(\T,x,y)=y$. 

Let $z\in D$, the solutions $\Phi^{\pm}(t,\T,z;\e)$ of \eqref{s1b}, restricted to $y\gtrless 0$, such that $\Phi^{\pm}(0,\T,z;\e)=(\T,z)$ are given as 
\begin{equation*}\label{Phisol}
\Phi^{\pm}(t,\T, z;\e)=\left(t+\T,\xi^{\pm}(t,\T, z;\e)\right),
\end{equation*}
where $\xi^{\pm}(t,\T,z;\e)$ are solutions of
\begin{equation}\label{auxs1}
\xi'=X^{\pm}_{\e}(t+\T,\xi), \quad \xi(0)=z, \quad \xi\in D.
\end{equation}

\begin{lemma}\label{fl}
Fix $T>0$, $\T\in\s^1_{\si}$, $z_0\in D$, and $z_1\in \R^2$. Let
\begin{equation}\label{psi}
\psi^{\pm}(t,\T,z_0,z_1)=Y^{\pm}(t,z_0)\left(z_1+\int_0^tY^{\pm}(s,z_0)^{-1}G^{\pm}\left(s+\T,\G^{\pm}(s,z_0)\right)ds\right),
\end{equation}
where $Y^{\pm}$ are the fundamental solutions \eqref{Ycolum} of the variational equations \eqref{vareq}. 
Then, for $\e>0$ small enough,  $z_0+\e z_1\in D$  and the next equality holds 
 $$\xi^{\pm}(t,\T_0+\e \T_1,z_0+\e z_1;\e)=\G^{\pm}(t,z_0)+\e \psi^{\pm}(t,\T_0,z_0,z_1)+\CO(\e^2),\, t\in[-T,T]$$
\end{lemma}
\begin{proof}
Computing the derivative in the variable $t$ in both sides of the equality $\xi^{\pm}(t,\T_0+\e \T_1,z_0+\e z_1;\e)=\G^{\pm}(t,z_0)+\e \Psi^{\pm}(t)+\CO(\e^2)$ we obtain
\[
\begin{array}{L}
F^{\pm}\left(\xi^{\pm}(t,\T_0+\e\T_1,z_0+\e z_1;\e)\right)+\e G^{\pm}\left(t+\T_0+\e\T_1,\xi^{\pm}(t,\T_0+\e\T_1,z_0+\e z_1;\e)\right)=\vspace{0.2cm}\\
F\left(\G^{\pm}(t,z_0)\right)+\e\dfrac{\p \Psi^{\pm}}{\p t}(t)+\CO(\e^2).
\end{array}
\]  
Expanding in Taylor series the lefthand side of the above equation around $\e=0$, and comparing the coefficient of $\e$ in the both sides, we conclude that
\[
\dfrac{\p \psi^{\pm}}{\p t}(t,\T_0,z_0,z_1)=DF^{\pm}\left(\G^{\pm}(t,z_0)\right)\Psi^{\pm}(t)+G^{\pm}\left(t+\T_0,\G^{\pm}(t,z_0)\right).
\] 
Moreover, $\psi^{\pm}(0,\T_0,z_0,z_1)=z_1$. 
Hence, the solution of the above differential equation is given by \eqref{psi}. 
We observe that $\Psi^{\pm}(t)$ depends on $\T_0,z_0,z_1$, then we denote  $\Psi^{\pm}(t)=\psi^{\pm}(t,\T_0,z_0,z_1)$. 
\end{proof}

Applying Lemma \ref{l1} to the fundamental matrices $Y^\pm$ (see \eqref{Ycolum}) in the expression \eqref{psi} we get
\[
\begin{array}{L}
\psi^{-}(t,\T,z_0,z_1)= Y(t,z_0)\left(z_1+\int_0^tY(s,z_0)^{-1}G^-\left(s+\T,\G(s,z_0)\right)ds\right),\vspace{0.2cm}\\
\psi^{+}(-t,\T,Rz_0,Rz_1)= RY(t,z_0)\left(z_1-\int_0^t Y(s,z_0)^{-1}RG^{+}\left(-s+\T,R\G(s,z_0)\right)ds\right).
\end{array}
\]
Moreover, using that $Y(t,z)=D_z \G(t,z)$ in the first part of the above expressions we have
\begin{equation}\label{psi2}
\begin{array}{L}
\psi_i^{-}(t,\T,z_0,z_1)=\vspace{0.2cm}\\
\quad\Big\langle D_z \Gamma_i(t,z_0)\,,z_1+\int_0^{t}
Y(s,z_0)^{-1}G^{-}\left(s+\T,\G(s,x,0)\right)ds\Big\rangle,\vspace{0.2cm}\\

\psi_i^{+}(-t,\T,Rz_0,Rz_1)=\vspace{0.2cm}\\
\quad(-1)^{i+1}\Big\langle  D_z \Gamma_i(t,z_0)\,,z_1-
\int_0^{t}Y(s,z_0)^{-1}RG^{+}\left(-s+\T,R\G(s,x,0)\right)ds\Big\rangle.\\
\end{array}
\end{equation}
for $i=1,2$.

Observe that for $\e=0$ system \eqref{s1b} has two lines of two-fold points, one invisible $(\T,x_i,0)$ and one visible $(\T,x_v,0)$ (see Figure \ref{fig2}).  
Moreover, for each $\T\in\s^1_{\si}$ and $x_i<x\leq x_v$ there exists a $2\ov\si(x)$--periodic solution $\widetilde{\G}(t,\T,x)=\big(t+\T\,,\,\g(t,x)\big)$, where 
$\gamma$ is given in \eqref{gamma}.
\begin{figure}[h]
\begin{center}
\begin{overpic}[width=12cm]{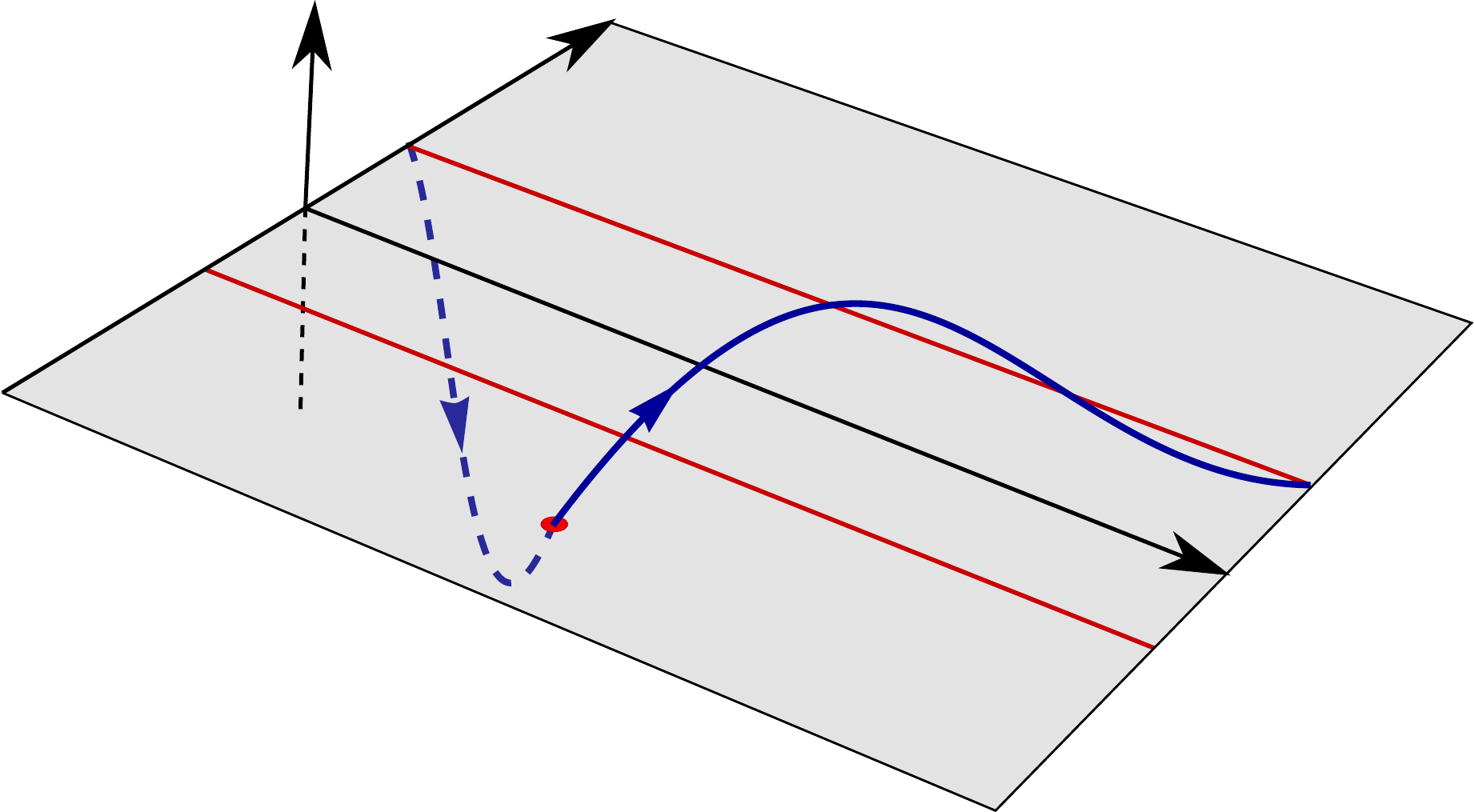}
%\begin{overpic}[grid,tics=5,width=14cm]{3DNPFigure.pdf}
\put(5,28){$\Sigma$} 
\put(38,15){$\widetilde\G(\si_v,\T_0,p_v)$}
\put(31,45){$(\T_0,p_v)$}
\put(11,38){$x_i$}
\put(23.5,46){$x_v$}
\put(84,15){$\T$}
\put(42,54){$x$}
\put(21,56){$y$}
\end{overpic}
\end{center}

\bigskip

\caption{The $2\si_v$--periodic solution $\widetilde\G(t,\T_0,p_v)$ of the extended system \eqref{s1b}, 
for $\e=0$, passing through the visible two-fold point $(\T_0,p_v)$.}\label{fig2}
\end{figure}

Notice that studying the bifurcation of the fold lines of system \eqref{s1b}, for $\e>0$, is equivalent to study the zeros of the functions
\[
\langle \nabla h(x,0), X^{\pm}_{\e}(\T,(x,0)) \rangle =X_2^{\pm}(\T,(x,0);\e)=F_2^{\pm}(x,0)+\e G^{\pm}_2(\T,x,0)+ \CO(\e^2).
\]
Thus,  by hypothesis (h1), we obtain that for $\e>0$ sufficiently small each one of the lines of visible-visible fold points $(\T,x_{v},0)$ bifurcates in two lines 
$(\T,\ell_v^{\pm}(\T;\e),0)$, one of visible fold points for $X^+_{\e}$ and another of visible fold points for $X^-_{\e}$. 
Analogously, the line of invisible-invisible fold points $(\T,x_{i},0)$ bifurcates in two lines $(\T,\ell_i^{\pm}(\T;\e),0)$, 
one of invisible fold points for $X^+_{\e}$ and another of invisible fold points for $X^-_{\e}$. 
Furthermore,
\begin{equation}\label{fold}
\begin{array}{l}
\ell_{v}^{\pm}(\T;\e)=x_{v}-\e \dfrac{G_2^{\pm}(\T,p_{v})}{\dfrac{\p F_2}{\p x}(p_{v})}+\CO(\e^2)=x_{v}+\e \nu^{\pm}_{v}(\T)+\CO(\e^2),\quad \text{and}\vspace{0.2cm}\\
\ell_{i}^{\pm}(\T;\e)=x_{i}-\e \dfrac{G_2^{\pm}(\T,p_{i})}{\dfrac{\p F_2}{\p x}(p_{i})}+\CO(\e^2)=x_{i}+\e \nu^{\pm}_{i}(\T)+\CO(\e^2).
\end{array}
\end{equation}

In what follows, $\pi_{\T}$,  $\pi_x$, and $\pi_y$ will denote the projections, defined on $\s^1_{\si}\times D,$ onto the first, second, and third coordinates, respectively.

\subsection{Proof of Theorem \ref{ta}}
The idea of this proof is to define a function $\CF:\s^1_{\si}\times(x_i,x_v)\rightarrow\R^2$ which allows us to determine the existence of crossing periodic solutions. 
Given $\T\in\s^1_{\si}$ and $x\in(x_i,x_v),$ we consider the flows $\Phi^-\big(t,\T,(x,0);\e\big)$ and $\Phi^+\big(t,\T+2\si,(x,0);\e\big)$ of \eqref{s1b}. 
If for some $\T_*\in\s^1_{\si}$ and $x_*\in(x_i,x_v)$ there exist $t_*^{-}\ge 0$ and  $t_*^{+}\le 0$ such that 
\begin{equation}\label{condicionesphix*}
\Phi^-\big(t_*^-,\T_*,(x_*,0);\e\big)=\Phi^+\big(t_*^+,\T_*+2\si,(x_*,0);\e\big)\in \widetilde{\Sigma},
\end{equation}
then $t_*^+=t_*^--2\si$ and therefore
\[
\Phi\big(t,\T_*,(x_*,0);\e\big)=\left\{\begin{array}{ll}
\Phi^+\big(t,\T_*+2\si,(x_*,0);\e\big)&\text{if}\quad t_*^+=t_*^--2\si\leq t\leq 0,\vspace{0.1cm}\\
\Phi^-\big(t,\T_*,(x_*,0);\e\big)& \text{if}\quad0\leq t\leq t_*^-
\end{array}\right.
\]
is a $2\si$--periodic crossing solution of system \eqref{s1b}. 
Indeed, this solution is well defined because 
\[
\begin{array}{rl}
\Phi^+\big(0,\T_*+2\si,(x_*,0);\e\big)=&\big(\T_*+2\si,\xi^{+}(0,\T_*+2\si,(x_*,0);\e)\big)\vspace{0.1cm}\\
=&
\big(\T_*+2\si,x_*,0\big),\vspace{0.2cm}\\
\Phi^-\big(0,\T_*,(x_*,0);\e\big)=&\big(\T_*,\xi^{-}(0,\T_*,(x_*,0);\e)\big)=\big(\T_*,x_*,0\big)
\end{array}
\]
and, as we are working in the cylinder $\s^1_{\si}\times D$, these two points are the same. 

In what follows, we show the existence of $\T_*$ and $x_*$ satisfying \eqref{condicionesphix*}.
For $\e=0$ we know that (see \eqref{gamma}) 
$$
\pi_y\Phi^-\big(\ov\si(x),\T,(x,0);0\big)=\xi_2^-\big(\ov\si(x),\T,(x,0);0\big)=
\Gamma_2 (\ov\si(x),x,0)= 0.
$$ 
Since, by  hypothesis $(h2)$, this flow reaches transversally the set of discontinuity $\widetilde\Sigma$ we can apply the {\it implicit function theorem} to obtain a time 
$t^-(\T,x;\e)=\ov\si(x)+\e\,t_1^-(\T,x)+\CO(\e^2)>0$ such that 
$$
\pi_y\Phi^-\big(t^-(\T,x;\e),\T,(x,0);\e\big)=\xi^-_2\big(t^-(\T,x;\e),\T,(x,0);\e\big)=0.
$$
Analogously, 
$$
\pi_y\Phi^+\big(-\ov\si(x),\T+2\si,(x,0);0\big)=\xi_2^+\big(-\ov\si(x),\T+2\si,(x,0);0\big)= -\Gamma_2 (\ov\si(x),x,0)= 0,
$$ 
therefore there exists $t^+(\T,x;\e)=-\ov\si(x)+\e\,t_1^+(\T,x)<0$ such that 
$$
\pi_y\Phi^+\big(t^+(\T,x;\e),\T+2\si,(x,0);\e\big)=\xi^+_2\big(t^+(\T,x;\e),\T+2\si,(x,0);\e\big)=0.
$$

Using the expression for $\xi^{\pm}_2$ given in Lemma \ref{fl} 
we can easily obtain that
\begin{equation}\label{ti}
t_1^{-}(\T,x)=-\dfrac{\psi_2^{-}(+ \ov\si(x),\T,(x,0),(0,0))}{F_2\left(\gamma(\ov\si(x),x,0)\right)},
\end{equation}
and
\begin{equation}\label{ti+}
\begin{array}{rl}
t_1^{+}(\T,x)=&-\dfrac{\psi_2^{+}(- \ov\si(x),\T+2\si,(x,0),(0,0))}{F_2\left(\gamma(\ov\si(x),x,0)\right)}\vspace{0.2cm}\\
=&-\dfrac{\psi_2^{+}(- \ov\si(x),\T,(x,0),(0,0))}{F_2\left(\gamma(\ov\si(x),x,0)\right)},
\end{array}
\end{equation}
where $\gamma$ is defined in \eqref{gamma}. 
Moreover, from \eqref{psi2}, we get
\begin{equation}\label{psii}
\begin{array}{L}
\psi_i^{-}(\ov\si(x),\T,(x,0),(0,0))=\vspace{0.2cm}\\
\quad\Big\langle D_z \Gamma_i(\ov\si(x),x,0)\,,\int_0^{\ov\si(x)}
\!\!\!\!\!\!\!\! Y(s,x,0)^{-1}G^{-}\left(s+\T,\g(s,x)\right)ds\Big\rangle,\vspace{0.3cm}\\
\psi_i^{+}(-\ov\si(x),\T+2\si,(x,0),(0,0))=\vspace{0.2cm}\\
\quad(-1)^i\Big\langle  D_z \Gamma_i(\ov\si(x),x,0)\,,
\int_0^{\ov\si(x)}
\!\!\!\!\!\!\!\! Y(s,x,0)^{-1}RG^{+}\left(-s+\T,R\g(s,x)\right)ds\Big\rangle,
\end{array}
\end{equation}
for $i=1,2$.

Accordingly, define $\CF(\T,x;\e)=(\CF_1(\T,x;\e),\CF_2(\T,x;\e))$ as
\[
\begin{array}{rl}
\CF_1(\T,x;\e)=&\pi_{\T}\Phi^-(t^-(\T,x;\e),\T,(x,0);\e)-\pi_{\T}\Phi^+(t^+(\T,x;\e),\T+2\si,(x,0);\e)\vspace{0.2cm}\\
=&t^-(\T,x;\e)-t^+(\T,x;\e)-2\si=2(\ov\si(x)-\si)+\CO(\e),\vspace{0.3cm}\\
\CF_2(\T,x;\e)=&\pi_{x}\Phi^-(t^-(\T,x;\e),\T,(x,0);\e)-\pi_{x}\Phi^+(t^+(\T,x;\e),\T+2\si,(x,0);\e)\vspace{0.2cm}\\
=&\xi_1^{-}(t^-(\T,x;\e),\T,(x,0);\e)-\xi_1^{+}(t^+(\T,x;\e),\T+2\si,(x,0);\e).
\end{array}
\]
From Lemma \ref{fl}, expressions \eqref{ti} and \eqref{ti+}, the reversibility condition \eqref{eq:F+-}, and using that 
$\g(-\ov\si(x),x)=\g(\ov\si(x),x)$ for $x_i<x<x_v$, we get
\[
\begin{array}{L}
\xi_1^{+}(t^{+}(\T,x;\e),\T+2\si,(x,0);\e)\vspace{0.2cm}\\
\quad=\g_1(\ov\si(x),x)+\e \Big(F_1^{+}(\g(\ov\si(x),x))t_1^{+}(\T,x)
+\psi_1^{+}(-\ov\si(x),\T,(x,0),(0,0))\Big)+\CO(\e^2)\vspace{0.2cm}\\
\quad=\g_1(\ov\si(x),x)+\e \Bigg(\dfrac{F_1(\g(\ov\si(x),x))}{F_2(\g(\ov\si(x),x))}\psi_2^{+}(-\ov\si(x),\T+2\si,(x,0),(0,0))\vspace{0.2cm}\\
\quad\quad+ \psi_1^{+}(-\ov\si(x),\T,(x,0),(0,0))\Bigg)+\CO(\e^2),\quad \text{and}\\
\end{array}
\]
\[
\begin{array}{L}
\xi_1^{-}(t^{-}(\T,x;\e),\T,(x,0);\e)\vspace{0.2cm}\\
\quad=\g_1(\ov\si(x),x)+\e \Big(F_1^{-}(\g(\ov\si(x),x))t_1^{-}(\T,x)
+\psi_1^{-}(\ov\si(x),\T,(x,0),(0,0))\Big)+\CO(\e^2)\vspace{0.2cm}\\
\quad=\g_1(\ov\si(x),x)+\e \Bigg(-\dfrac{F_1(\g(\ov\si(x),x))}{F_2(\g(\ov\si(x),x))}\psi_2^{-}(\ov\si(x),\T,(x,0),(0,0))\vspace{0.2cm}\\
\quad\quad+ \psi_1^{-}(\ov\si(x),\T,(x,0),(0,0))\Bigg)+\CO(\e^2).\\
\end{array}
\]
Therefore, 
\[
\begin{array}{rl}
\dfrac{\CF_2(\T,x;\e)}{\e}=&\psi_1^{-}(\ov\si(x),\T,(x,0),(0,0))-\psi_1^{+}(-\ov\si(x),\T+2\si,(x,0),(0,0))\vspace{0.2cm}\\
&-\dfrac{F_1(\g(\ov\si(x),x))}{F_2(\g(\ov\si(x),x))}\Big(\psi_2^{+}(-\ov\si(x),\T+2\si,(x,0),(0,0))\vspace{0.2cm}\\
&+\psi_2^{-}(\ov\si(x),\T,(x,0),(0,0))\Big)+\CO(\e).
\end{array}
\]
Now, from \eqref{psii} we have that
\[
\begin{array}{L}
\psi_1^{-}(\ov\si(x),\T,(x,0),(0,0))-\psi_1^{+}(-\ov\si(x),\T+2\si,(x,0),(0,0))=\vspace{0.2cm}\\
\quad \Bigg\langle D_z\G_1(\ov\si(x),x,0)\,,\,\int_0^{\ov\si(x)}Y(t,x,0)^{-1}\big\{G^-,G^+\big\}_{\T}(t,\g(t,x))dt\Bigg\rangle,
\end{array}
\]
and 
\[
\begin{array}{L}
\psi_2^{-}(\ov\si(x),\T,(x,0),(0,0))+\psi_2^{+}(-\ov\si(x),\T+2\si,(x,0),(0,0))=\vspace{0.2cm}\\
\quad \Bigg\langle D_z\G_2(\ov\si(x),x,0)\,,\,\int_0^{\ov\si(x)}Y(t,x,0)^{-1}\big\{G^-,G^+\big\}_{\T}(t,\g(t,x))dt\Bigg\rangle,
\end{array}
\]
where $\{G^-,G^+\}_{\T}(t,z) = G^-(t+\T,z)+RG^+(-t+\T,Rz)$, see \eqref{rever}.

Since $Y^T\big(-F_2,F_1\big)^T=F_1D_z\G_2-F_2D_z\G_1$, we obtain
\[
\langle F_1D_z\G_1-F_2D_z\G_2\,,\,V\rangle=\big\langle \big(-F_2,F_1\big)\,,\,YV\big\rangle=F\wedge YV.
\]
Hence, we conclude that
\begin{equation}\label{melnikov1}
-F_2(\g(\ov\si(x),x))\CF_2(\T,x;\e)=\e M(\T,x)+\CO(\e^2).
\end{equation}
where $M(\T,x)$ is defined in \eqref{mel}.

From the construction of $\CF$ it is clear that a subharmonic crossing periodic solution of system \eqref{ps1} exists, for $\e>0$ sufficiently small, 
if, and only if, there are $\T_{\e}\in\s^1_{\si}$ and $x_{\e}\in(x_i,x_v)$ such that $\CF(\T_{\e},x_{\e};\e)=(0,0)$.

By hypothesis $\CF_1(\T,x_\si;0)=0$ and, from \eqref{eq:sigmaprima}, 
$$
\frac{\p\CF_1}{\p x}(\T,x_\si;0)=2 \ov\si'(x_\si)=\frac{\dfrac{\p \g_2}{\p x}(\ov\si(x_\si),x_\si)  }{F_2(\g(\ov\si(x_\si),x_\si)) }\neq0.
$$ 
Thus, by the {\it Implicit Function Theorem} there exists $x(\T;\e)$ such that $\CF_1(\T,x(\T;\e);\e)=0$ and $x(\T;\e)\to x_\si$
when $\e\to 0$ for every $\T\in\s^1_{\si}$.

Now, we take 
\[
\widetilde\CF(\T;\e)=-\dfrac{F^{+}_2\big(\g(\ov\si(x(\T;\e)),x(\T;\e))\big)}{\e}\CF_2(\T,x(\T;\e);\e).
\] 
From \eqref{melnikov1} the above equation reads
\[
\widetilde\CF(\T;\e)=M(\T,x(\T;\e))+\CO(\e)=M(\T,x_\si)+\CO(\e),
\]
By hypothesis there exists $\T^*\in\s^1_{\si}$ such that $\widetilde\CF(\T^*,0)=M(\T^*,x_\si)=0$ and 
$(\p \widetilde\CF/\p\T)(\T^*,0)=(\p M/\p\T)(\T^*,x_\si)\neq0$. 
Thus, applying again the {\it Implicit Function Theorem} we conclude that, for $\e>0$ sufficiently small, 
there exists $\T_\e\in \s^1_{\si}$ such that $\widetilde\CF(\T_\e;\e)=0$. Moreover, $\T_\e\to\T^*$ as $\e\to 0$. 
This concludes the proof of Theorem \ref{ta}. $\Box$

\subsection{Proof of statement {\bf (a)} of Theorem \ref{tb}}
Since $G^+_2(\T^*,p_v)\neq G^-_2(\T^*,p_v),$ we can assume that there exists $a,b\in[0,2\si_v]$ with $a<b$ and 
$\T^*\in(a,b)$ such that $G^+_2(t,p_v)\neq G^-_2(t,p_v)$ for every $t\in[a,b]$. Without loss of generality, we suppose that $G^+_2(t,p_v)<G^-_2(t,p_v)$ for every $t\in[a,b]$. At the end of the proof we shall comment the case when $G^+_2(t,p_v)>G^-_2(t,p_v)$ for every $t\in[a,b]$.

The above assumption and expression \eqref{fold} imply that $\ell_v^{+}(\T;\e)>\ell_v^{-}(\T;\e)$ for every $\T\in[a,b]$ 
and $\e>0$ sufficiently small.

Let $\CR_{\e}$ be the region on $\widetilde{\Sigma}\subset \s^1_{\si_v}\times\R$ delimited by the graphs 
$\ell_v^{\pm}(\T;\e)$ for $\T\in[a,b]$, that is $\CR_{\e}=\{(\T,x,0):\,\T\in[a,b],\,\ell_v^{-}(\T;\e)<x<\ell_v^{+}(\T;\e)\}$. 
A straightforward computation shows that this is a region of sliding type.  
Moreover, the autonomous vector field \eqref{s1b} is $2\si_v$--periodic in the variable $\T$, so the regions 
$\CR_{\e}^{n}=\{(\T+2n\si_v,x):\,(\T,x)\in\CR_{\e}\}$ for $n\in\N$ are of sliding type.

The expression of the sliding vector field for each region
 $\CR^{n}_{\e}$, $n\in\N$, is
\begin{equation}\label{ss}
\begin{array}{rll}
\T'=&\!\!\!u(\T,x;\e)=&\!\!\! 1,\vspace{0.2cm}\\
\e x'=&\!\!\!v(\T,x;\e)=&\!\!\!\dfrac{f_0(x)}{G^+_2(\T,x,0)-G^-_2(\T,x,0)}+\e\Bigg(\dfrac{f_1(\T,x)}{G^+_2(\T,x,0)-G^-_2(\T,x,0)}\vspace{0.2cm}\\
&&\!\!\!+f_0(x)\dfrac{H^-_2(\T,(x,0);\e)-H^+_2(\T,(x,0);\e)}{(G^+_2(\T,x,0)-G^-_2(\T,x,0))^2}\Bigg)+\CO(\e^2),
\end{array}
\end{equation}

where
\[
\begin{array}{rl}
f_0(x)=&2F_1(x,0)F_2(x,0),\vspace{0.3cm}\\
f_1(\T,x)=&F_2(x,0)\left(G^-_1(\T,x,0)-G^+_1(\T,x,0)\right)\vspace{0.2cm}\\
&+F_1(x,0)\left(G^+_2(\T,x,0)+G^-_2(\T,x,0)\right).
\end{array}
\]

System \eqref{ss} can be studied using singular perturbation theory  (see for instance \cite{F,J}). 
In this theory, system \eqref{ss} is known as {\it slow system}. 
Doing $\e=0$ we can find the {\it critical manifold} as 
\[
\CM_0=\{(\T,x)\in\CR_{\e}^{n}:\, f_0(x)=0\}=\{(\T+2n\si_v,x_v):\, \T\in[a,b]\}.
\]

Now, doing the time rescaling $t=\e \tau$, system \eqref{ss}, for $\e>0,$ becomes
\begin{equation}\label{fs}
\begin{array}{l}
\dot{\T}=\e u(\T,x;\e)=\e,\vspace{0.2cm}\\
\dot x=v(\T,x;\e),
\end{array}
\end{equation}
which is known as {\it fast system}. 
Computing the derivative with respect to the variable $x$ of the function $v$ for $\e=0$ 
at the points of $(\T+2n\si_v,x_v)\in\CM_0$ we obtain
\begin{equation}\label{ptheta}
\dfrac{\p v}{\p x}(\T+2n\si_v,x_v;0)=p(\T)=\dfrac{2F_1(p_v)\dfrac{\p F_2}{\p x}(p_v)}{G^+_2(\T,p_v)-G^-_2(\T,p_v)}>0,
\end{equation}
for every $\T\in[a,b]$, by hypothesis (h2) and the assumption $G^+_2(\T,p_v)<G^-_2(\T,p_v)$.
Therefore, for $\e=0$, $\CM_0$ is a normally hyperbolic repelling critical manifold for the vector field \eqref{fs} and also for the sliding vector field \eqref{ss}.

Applying {\it Fenichel's theorem} we conclude that there exists a normally hyperbolic repelling locally invariant manifold 
$\CM_{\e}=\{(\T+2n\si_v,m(\T;\e)):\,\T\in[a,b]\}$ of the system \eqref{fs},
which is $\e$--close to $\CM_0$:  
$$
m(\T;\e)=x_v+\e m_1(\T)+\CO(\e^2).
$$ 
Notice that $(\T(t)+2n\si_v,m(\T(t);\e))$ is a trajectory of system \eqref{ss}, so $v(\T+2n\si_v,m(\T;\e);\e)=\e\,(\p m/\p\T)(\T;\e).$ 
Accordingly, for $\e\geq0$ small enough, we may compute
\begin{equation}\label{m1}
m_1(\T)=-\dfrac{G^+_2(\T,p_v)+G^-_2(\T,p_v)}{2\dfrac{\p F_2}{\p x}(p_v)},
\end{equation}
therefore $\CM_{\e}\subset \CR_{\e}$.

Since the Fenichel's manifold is repelling we have that for a given point $(\T_0,\ell^-_v(\T_0;$ $\e))\in \p \CR^{n}_{\e}$ 
there exists a orbit $\delta(\T_0;\e)$ of the sliding vector field \eqref{ss} reaching the point $(\T_0,\ell^-_v(\T_0;\e))$ (see Figure \ref{slidingmethod}). 
In the sequel, we shall parametrize this orbit.

Given $N>0$, we want to compute the solution of system \eqref{ss} starting at $(\T,\ell^-_v(\T;\e))$, for $-N\e<t<0$.
equivalently, we compute the solution of system \eqref{fs} starting at the same point but for  $-N<\tau<0$. 

We denote by $\big(\T_s(\tau,\T;\e),x_s(\tau,\T;\e)\big)$ the solution of \eqref{fs} with initial condition: 
$\big(\T_s(0,\T;\e),x_s(0,\T;\e)\big)=(\T,\ell^-_v(\T;\e))$. 
Clearly $\T_s(\tau,\T;\e)=\T+\e \tau.$ Take $x_s(\tau,\T;\e)=x_v+\e k(\tau,\T)+\CO(\e^2)$. 
Expanding the both sides of the equality
\[
\dfrac{\p x_s}{\p \tau}(\tau,\T;\e)=v(\T+\e \tau,x_s(\tau,\T;\e);\e)
\]
in Taylor series respect to $\e$ we derive the following differential equation
\begin{equation}\label{k}
\begin{array}{RL}
\dfrac{\p k}{\p \tau}(\tau,\T)=&p(\T)k(\tau,\T)+F_1(p_v)\left(\dfrac{G_2^+(\T,p_v)+G_2^-(\T,p_v)}{G_2^+(\T,p_v)-G_2^-(\T,p_v)}\right)\vspace{0.2cm}\\
=&p(\T)k(\tau,\T)-m_1(\T)p(\T),\vspace{0.3cm}\\
k(0,\T)=& \nu_v^-(\T)=-\dfrac{G_2^{-}(\T,p_{v})}{\dfrac{\p F_2}{\p x}(p_{v})}=m_1(\T)+\dfrac{F_1(p_v)}{p(\T)},
\end{array}
\end{equation}
where $ \nu_v^-(\T),$  $p(\T),$ and $m_1(\T)$ are defined in \eqref{fold},  \eqref{ptheta}, and \eqref{m1}, respectively. 
The relation $p(\T)(\nu_v^-(\T)-m_1(\T))=F_1(p_v)$ has been used in order to get the above equalities. 
Solving the initial value problem  \eqref{k} we obtain
\begin{equation}\label{ksol}
k(\tau,\T)=m_1(\T)+\dfrac{F_1(p_v)}{p(\T)}e^{\tau p(\T)}.
\end{equation}

We have then found a set 
$$
\widetilde{\delta}(\T;\e)=\{\left(\T_s(\tau,\T;\e),x_s(\tau,\T;\e)\right):\,-N<\tau<0\}
$$ 
parametrized by $\tau,$ which is contained in the orbit $\delta(\T;\e)$.

From here, the idea of the proof is analogous to the proof of Theorem \ref{ta}, 
which consists in defining a function 
$\CF:(a,b)\times(-N,0)\rightarrow\R^2$ 
that allows us to determine the existence of sliding periodic solutions of system \eqref{ps1}. 
Given $\T\in(a,b),$ we consider the flows 
\[
\Phi^-\big(t,\T,\ell_v^-(\T;\e),0;\e\big)\,\, \text{and}\,\,
\Phi^+\big(t,\T_s(\tau, \T+2\si_v;\e),x_s(\tau, \T+2\si_v;\e),0;\e\big).
\]
The vector field \eqref{s1b} is $2\si_v$-- periodic in the variable $\T$, which means that $\T\equiv \T+2\si_v$. 
Thus, if for some $\T_*\in[0,2\si_v]$ and $\tau_*\in(-N,0)$ there exist $s_*^{-}\ge 0$ and $s_*^{+}\le 0$ such that 
$$
\Phi^-\big(s_*^{-},\T_*,\ell_v^-(\T_*;\e),0;\e\big)=
\Phi^+\big(s_*^{+},\T_s(\tau_*,\T_*+2\si_v;\e),x_s(\tau_*,\T_*+2\si_v;\e),0;\e\big)\in\Sigma,
$$ 
then there exists a sliding $2\si_v$--periodic solution of system \eqref{s1b} and, consequently, of system \eqref{ps1} (see Figure \ref{slidingmethod}).
\begin{figure}[h]
\begin{center}
\begin{overpic}[width=12cm]{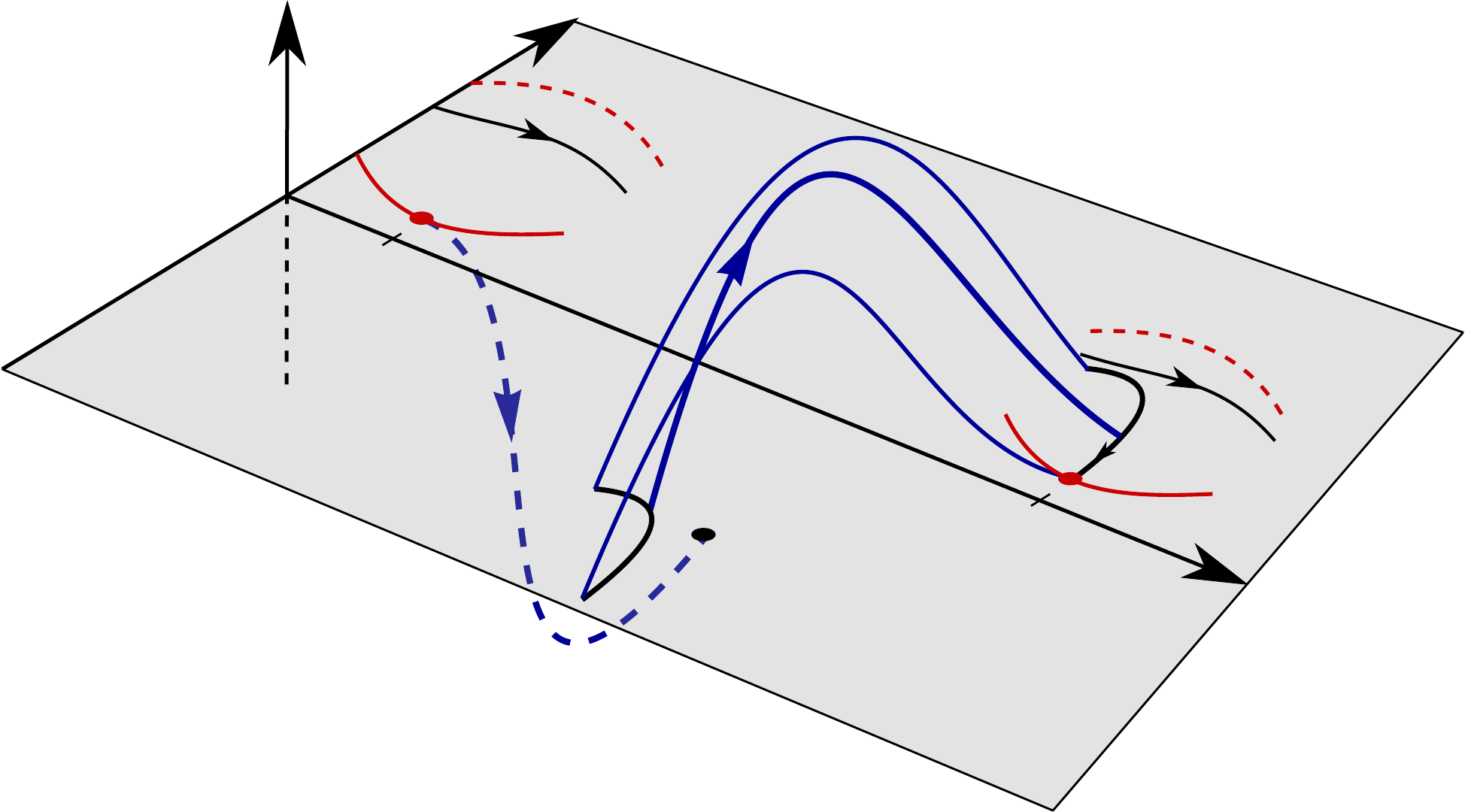}
%\begin{overpic}[grid,tics=5,width=14cm]{SlidingMethod.pdf}
\put(5,30){$\Sigma$} 
\put(28,50){$\ell_v^+$}
\put(21,46.5){$\ell_v^-$}
\put(86,14){$\T$}
\put(40,55){$x$}
\put(19.5,56.5){$y$}
\put(23,37){\rotatebox{-39}{$\T^*$}}
\put(41,39.5){$\CM_\e$}
\put(77.5,24){$\tilde\de\subset\de$}
\put(63,20){\rotatebox{-23}{$\T^*+2\si_v$}}
\end{overpic}
\end{center}

\bigskip

\caption{Methodology for constructing the displacement function to detect sliding $2\si_v$--periodic solutions of system \eqref{ps1}, 
for $\e>0$ sufficiently small.}\label{slidingmethod}
\end{figure}
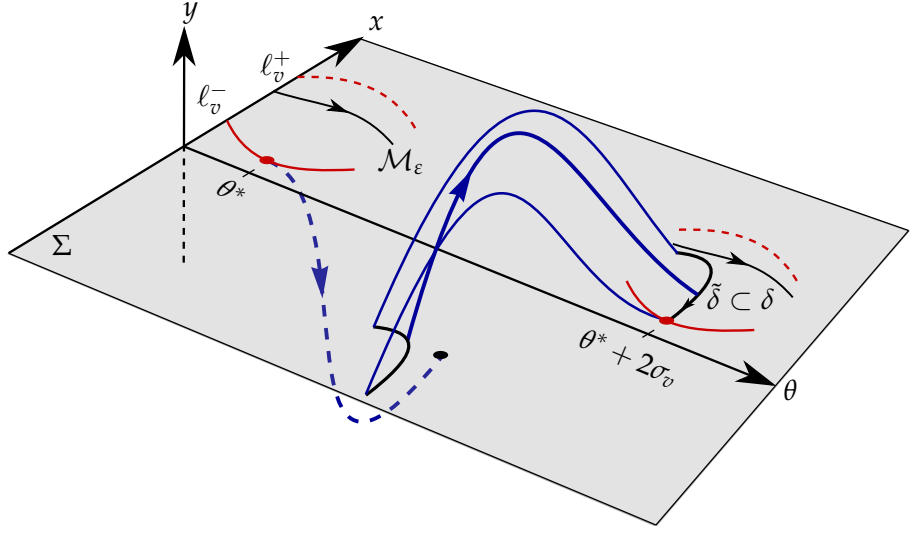

Again, analogously to the proof of Theorem \ref{ta}, we can use the {\it implicit function theorem} to find times 
$s^{-}(\T;\e)> 0$ and 
$s^{+}(\tau, \T;\e)< 0$
such that 
\[
\begin{array}{l}
\xi^-_2\big(s^-(\T;\e),\T,\ell_v^-(\T;\e),0;\e\big)=0 \quad \text{and} \vspace{0.2cm}\\
\xi^+_2\big(s^+(\tau, \T;\e),\T_s(\tau, \T+2\si_v;\e),x_s(\tau,\T+2\si_v;\e),0;\e\big)=0.
\end{array}
\]
Moreover, using the expression for $\xi^-_2$ given in Lemma \ref{fl}  with $\T=\T,$ $z_0=p_v$, and 
$z_1=(\nu_v^-(\T),0),$ where $\nu_v^-(\T)$ 
is given in \eqref{fold}, we obtain that $s^{-}(\T;\e)=\si_v+\e s_1^{-}(\T)+\CO(\e^2)$ provided that
\begin{equation}\label{simenos}
s_1^{-}(\T)=-\dfrac{\psi_2^{-}(\si_v,\T,p_v,(\nu_v^-(\T),0))}{F_2\left(q_v\right)},
\end{equation}
where $q_v$ is defined in \eqref{eq:qv}.
Analogously, using the expression for $\xi^+_2$ given in Lemma \ref{fl}  with 
$\T=\T+2\si_{v},$ $z_0=p_v$, and $z_1=(k(\tau,\T),0),$ where $k(\tau,\T)$ 
is given in \eqref{ksol}, we obtain 
$s^{+}(\tau, \T;\e)=-\si_v+\e s^{+}_1(\tau,\T)+\CO(\e^2)$ provided that
\begin{equation}\label{siplus}
s_1^{+}(\tau,\T)=-\dfrac{\psi_2^{+}(-\si_v,\T,p_v,(k(\tau,\T),0))}{F_2\left(q_v\right)}.
\end{equation}
Moreover, from \eqref{psi2}, we get
\begin{equation}\label{psii2}
\begin{array}{L}
\psi_i^{-}(\si_v,\T,p_v,(\nu_v^-(\T),0))=\dfrac{\p\Gamma_i}{\p x}(\si_v,p_v)\nu_v^-(\T)+\vspace{0.2cm}\\
\qquad \Big\langle D_z \Gamma_i(\si_v,p_v)\,,\,\int_0^{\si_v}
\!\!\!\!\!\! Y(s,p_v)^{-1}G^{-}\left(s+\T,\g(s,x_v)\right)ds\Big\rangle,\vspace{0.3cm}\\

\psi_i^{+}(-\si_v,\T,p_v,(k(\tau,\T),0))=(-1)^{i+1}\dfrac{\p\Gamma_i}{\p x}(\si_v,p_v)k(\tau,\T) +\vspace{0.2cm}\\
\qquad\Big\langle D_z \Gamma_i(\si_v,p_v)\,,(-1)^{i}\int_0^{\si_v}
\!\!\!\!\!\! Y(s,p_v)^{-1}RG^{+}\left(-s+\T,R\g(s,x_v)\right)ds\Big\rangle.
\end{array}
\end{equation}
for $i=1,2$. 

Accordingly, define $\CG(\tau,\T;\e)=(\CG_1(\tau,\T;\e),\CG_2(\tau,\T;\e))$ 
as
\[
\begin{array}{rl}
\CG_1(\tau,\T;\e)=&\pi_{\T}\Phi^-(s^-(\T;\e),\T,(\ell_v^-(\T;\e),0);\e)\vspace{0.2cm}\\
&-\pi_{\T}\Phi^+(s^{+}(\tau,\T;\e),\T_s(\tau,\T+2\si_v;\e),(x_s(\tau,\T+2\si_v;\e),0);\e)\vspace{0.3cm}\\
=&s^-(\T;\e)+\T-s^+(\tau,\T;\e)-\T_s(\tau,\T+2\si_v;\e)\vspace{0.3cm}\\
=&\e(s_1^-(\T)-s_1^+(\tau,\T)-\tau)+\CO(\e^2),\vspace{0.4cm}\\

\CG_2(\tau,\T;\e)=&\pi_{x}\Phi^-(s^-(\T;\e),\T,(\ell_v^-(\T;\e),0);\e)\vspace{0.2cm}\\
&-\pi_{x}\Phi^+(s^{+}(\tau,\T;\e),\T_s(\tau,\T+2\si_v;\e),(x_s(\tau,\T+2\si_v;\e),0);\e)\vspace{0.3cm}\\
=&\xi_1^-(s^-(\T;\e),\T,(x_{v}+\e \nu^{-}_{v}(\T),0);\e)\vspace{0.2cm}\\
&-\xi_1^+(s^{+}(\tau,\T;\e),\T+2\si_v+\e\tau,(x_v+\e k(\tau,\T),0);\e).\vspace{0.3cm}\\
\end{array}
\]
To compute the function $\CG_2$,  first we see that
\begin{equation*}\label{sms}
\begin{array}{L}
F_2(q_v)\big(s_1^-(\T)-s_1^+(\tau,\T)\big)\vspace{0.2cm}\\
\qquad=\psi_2^{+}(-\si_v,\T,p_v,(k(\tau,\T),0))-\psi_2^{-}(\si_v,\T,p_v,(\nu_v^-(\T),0))\vspace{0.2cm}\\
\qquad=-\dfrac{\p\Gamma_2}{\p x}(\si_v,p_v)\big(k(\tau,\T)+\nu_v^-(\T)\big)-g_\T\vspace{0.2cm}\\
\qquad=-\dfrac{F_2(q_v)}{F_1(p_v)}\big(k(\tau,\T)+\nu_v^-(\T)\big)-g_\T,
\end{array}
\end{equation*}
where $g_\T$ is defined in \eqref{g}. 
To obtain the above expression we have used Lemma \ref{l2} and expression \eqref{psii2}. Therefore,
\[
\CG_1(\tau,\T;\e)=-\dfrac{\e}{F_2(q_v)}\left(F_2(q_v)\tau +\dfrac{F_2(q_v)}{F_1(p_v)}\big(k(\tau,\T)+\nu_v^-(\T)\big)+g_\T\right)+\CO(\e^2).
\]

We compute $\CG_2(\tau,\T;\e).$ From Lemma \ref{fl} and expressions \eqref{simenos}, \eqref{siplus}, and \eqref{psii2} we get
\[
\begin{array}{L}
\xi_1^-(s^-(\T;\e),\T,(x_{v}+\e \nu^{-}_{v}(\T),0);\e)\vspace{0.3cm}\\
\quad=\g_1(\si_v,x_v)+\e \Big(F_1^{-}(q_v)s_1^{-}(\T)
+\psi_1^{-}(\si_v,\T,p_v,(\nu_v^-(\T),0))\Big)+\CO(\e^2)\vspace{0.3cm}\\
\quad=\g_1(\si_v,x_v)+\e \Bigg(-\dfrac{F_1(q_v)}{F_2(q_v)}\psi_2^{-}(\si_v,\T,p_v,(\nu_v^-(\T),0))\vspace{0.2cm}\\
\quad\quad+ \psi_1^{-}(\si_v,\T,p_v,(\nu_v^-(\T),0))\Bigg)+\CO(\e^2),\vspace{0.4cm}\\
\end{array}
\]
where $\gamma$ is given in \eqref{gamma}, and
\[
\begin{array}{L}
\xi_1^+(s^{+}(\tau,\T;\e),\T+2\si_v+\e\tau,(x_v+\e k(\tau,\T),0);\e)\vspace{0.3cm}\\
\quad=\g_1(\si_v,x_v)+\e \Big(F_1^{+}(q_v)s_1^{+}(\tau,\T)+\psi_1^{+}(-\si_v,\T,p_v,(k(\tau,\T),0))\Big)+\CO(\e^2)\vspace{0.3cm}\\
\quad=\g_1(\si_v,x_v)+\e \Bigg(\dfrac{F_1(q_v)}{F_2(q_v)}\psi_2^{+}(-\si_v,\T,p_v,(k(\tau,\T),0))\vspace{0.2cm}\\
\quad\quad+ \psi_1^{+}(-\si_v,\T,p_v,(k(\tau,\T),0))\Bigg)+\CO(\e^2).\vspace{0.4cm}\\
\end{array}
\]
Thus, 
\[
\begin{array}{L}
\dfrac{\CG_2(\tau,\T;\e)}{\e}=\psi_1^{-}(\si_v,\T,p_v,(\nu_v^-(\T),0))-\psi_1^{+}(-\si_v,\T,p_v,(k(\tau,\T),0))\vspace{0.2cm}\\
\quad-\dfrac{F_1(q_v)}{F_2(q_v)}\Big(\psi_2^{-}(\si_v,\T,p_v,(\nu_v^-(\T),0))+\psi_2^{+}(-\si_v,\T,p_v,(k(\tau,\T),0))\Big)+\CO(\e).
\end{array}
\]
From \eqref{psii2} we have
\[
\begin{array}{L}
\psi_1^{-}(\si_v,\T,p_v,(\nu_v^-(\T),0))-\psi_1^{+}(-\si_v,\T,p_v,(k(\tau,\T),0))=\vspace{0.2cm}\\
\qquad \dfrac{\p\G_1}{\p x}(\si_v,p_v)\big(\nu^-_v(\T)-k(\tau,\T)\big)+\vspace{0.2cm}\\
\qquad \Bigg\langle D_z\G_1(\si_v,p_v)\,,\,\int_0^{\si_v}Y(t,p_v)^{-1}\big\{G^-,G^+\big\}_{\T}(t,\g(t,x_v))dt\Bigg\rangle,
\end{array}
\]
and 
\[
\begin{array}{L}
\psi_2^{-}(\si_v,\T,p_v,(\nu_v^-(\T),0))+\psi_2^{+}(-\si_v,\T,p_v,(k(\tau,\T),0))=\vspace{0.2cm}\\
\qquad\dfrac{\p\G_2}{\p x}(\si_v,p_v)\big(\nu^-_v(\T)-k(\tau,\T)\big)+\vspace{0.2cm}\\
\qquad \Bigg\langle D_z\G_2(\si_v,p_v)\,,\,\int_0^{\si_v}Y(t,p_v)^{-1}\big\{G^-,G^+\big\}_{\T}(t,\g(t,x_v))dt\Bigg\rangle.
\end{array}
\]
Similar to the proof of Theorem \ref{ta} we obtain that
\[
-F_2(q_v)\CG_2(\tau,\T;\e)=\e \big(\nu^-_v(\T)-k(\tau,\T)\big)F(q_v)\wedge\dfrac{\p\Gamma}{\p x}(\si_v,p_v)+ \e M(\T,x_v)+\CO(\e^2),
\]
where $M(\T,x)$ is defined in \eqref{mel}. As a direct consequence of Lemma \ref{l2} we have that the above wedge product vanishes. Therefore,
\begin{equation*}\label{melnikov2}
\CG_2(\tau,\T;\e)=-\dfrac{\e}{F_2(q_v)}  M(\T,x_v)+\CO(\e^2).
\end{equation*}

Now, consider the function
\[
\widetilde\CG(\tau,\T;\e)=-\dfrac{F_2(q_v)}{\e}\CG(\tau,\T;\e)=\Big(\widetilde\CG_1(\tau,\T)\,,\,\widetilde\CG_2(\T)\Big)+\CO(\e).
\]
Thus,
\[
\begin{array}{RL}
\widetilde\CG_1(\tau,\T)=&F_2(q_v)\tau+\dfrac{F_2(q_v)}{F_1(p_v)}\big(k(\tau,\T)+\nu_v^-(\T)\big)+g_\T,\vspace{0.2cm}\\
\widetilde\CG_2(\T)=&M(\T,x_v).
\end{array}
\]
By hypothesis there exists $\T^*\in\s^1_{\si_v}$ such that $M(\T^*,x_v)=0$ and $(\p M/\p\T)$ $(\T^*,x_v)\neq 0$. 
Now, we note that the equation $\widetilde\CG_1(\tau,\T^*)=0$ is equivalent, using \eqref{ksol},  to the equation
\begin{equation}\label{e1}
\tau+\dfrac{1}{p(\T^*)}e^{\tau p(\T^*)}+A(\T^*)=0 \quad \text{where}\quad A(\T)=\dfrac{m_1(\T)+\nu_v^-(\T)}{F_1(p_v)}+\dfrac{g_\T}{F_2(q_v)},
\end{equation}
where $p(\T)$ and $m_1(\T)$ are defined in \eqref{ptheta} and \eqref{m1}, respectively. Since $p(\T^*)>0$, equation \eqref{e1} becomes
\[
r(\tau)e^{r(\tau)}=e^{-A(\T^*)p(\T^*)}\quad \text{with}\quad r(\tau)=-(\tau+A(\T^*))p(\T^*),
\]
which admits a unique real solution
\[
\tau^*=-A(\T^*)-\dfrac{1}{p(\T^*)}W\left(e^{-A(\T^*)p(\T^*)}\right).
\]
Here, $W$ denotes the Lambert W--function ($x=W(y)$  gives the solution of $ x e^x=y$, for a definition see \cite{C}).
From the properties of the W--function, we know that $W(e^{\beta})>\beta$ if, and only if, $\beta<1$. 
Then, we obtain that $\tau^*<0$ if, and only if, $A(\T^*)p(\T^*)>-1$. 
This follows from the hypothesis {\bf (a)} of the theorem, which reads 
\[
g_{\T^*}>\dfrac{2F_2(q_v)}{F_1(p_v)\dfrac{\p F_2}{\p x}(p_v)}G_2^-(\T^*,p_v).
\]

Accordingly, we take $N=-2\tau^*$ in order to have $(\tau ^*,\T^*)\in (-N,0)\times (a,b)$ and $\widetilde\CG(\tau ^*, \T^*,0)=0$. 
Moreover,
\[
\begin{array}{RL}
\det\left(\dfrac{\p\widetilde\CG}{\p(\tau,\T)}(\tau ^*,\T^*,0)\right)=
&\left|\begin{array}{cc}   F_2(q_v)\left(1+e^{\tau^*p(\T^*)}\right) & ^\#\\
0 & \dfrac{\p M}{\p\T} (\T^*,x_v) \end{array}\right|\vspace{0.2cm}\\
=&F_2(q_v)\left(1+e^{\tau^*p(\T^*)}\right)\dfrac{\p M}{\p\T}(\T^*,x_v)\neq 0.
\end{array}
\]
Thus, applying the {\it implicit function theorem}  we conclude that for $\e>0$ sufficiently small there exist 
$\T_\e\in(a,b)$ and $\tau_\e\in(-N,0)$ 
such that $\CG(\tau_\e,\T_\e;\e)=\widetilde\CG(\tau_\e,\T_\e;\e)=0$ and $\T_\e\to\T^*$ and $\tau_\e\to\tau^*$ when $\e\to 0$.

For the case when $G^+_2(t,p_v)>G^-_2(t,p_v)$ for every $t\in[a,b]$ the same argument works  reversing time. 
Therefore, in this case,  the obtained sliding periodic solutions slide on $\Sigma^e$. This concludes the proof of item {\bf (a)} of the Theorem \ref{tb}. $\Box$

\subsection{Proof of statement {\bf (b)} of Theorem \ref{tb}}
Let $K\subset \s^1_{\si_v}\times\R$ be the set of pairs $(\T,\chi)$ such that $\chi<\min\{\nu^{\pm}(\T)\}$. 
Clearly, in this case, $\zeta_\e=x_v+\e \chi<\ell_v^{\pm}(\T;\e)$, and therefore the solutions of system \eqref{auxs1} cross the set of discontinuity 
$\widetilde\Sigma$ at the points $(\T,(\zeta_\e,0))$ when $(\T,\chi)\in K$. 

In what follows, we define a function $\CH:K\times(0;\e_0)\rightarrow\R^2$ such that its zeros determine the existence of crossing periodic solutions near 
the separatrix $\CS$. 
Given $(\T,\zeta_\e)\in K,$ we consider the flows $\Phi^-\big(t,\T,(\zeta_\e,0);\e\big)$ and $\Phi^+\big(t,\T+2\si_v,(\zeta_\e,0);\e\big)$. 
The existence of times $r^-=r^-(\T,\chi;\e)>0$ and $r^+=r^+(\T,\chi;\e)<0$  such that 
$$
\xi^-_2\big(r^-,\T,(\zeta_\e,0);\e\big)=0, \quad \xi^+_2\big(r^+,\T+2\si_v,(\zeta_{\e},0);\e\big)=0
$$ 
is guaranteed by the {\it implicit function theorem}. 
These times can be computed analogously to \eqref{simenos} and \eqref{siplus} as 
$$
r^{-}(\T,\chi;\e)=\si_v+\e r_1^{-}(\T,\chi)+\CO(\e^2)\quad
r^{+}(\T,\chi;\e)=-\si_v+\e r_1^{+}(\T,\chi)+\CO(\e^2),
$$ 
where
\begin{equation*}\label{ri}
r_1^{\pm}(\T,\chi)=-\dfrac{\psi_2^{\pm}(\mp\si_v,\T,p_v,(\chi,0))}{F_2(q_v)},
\end{equation*}
but here we have used the formula of Lemma \ref{fl} for $z_0=p_v$, and $z_1=(\chi,0).$ 

Accordingly, define $\CH(\T,\chi;\e)=(\CH_1(\T,\chi;\e),\CH_2(\T,x;\e))$ as
\[
\begin{array}{rl}
\CH_1(\T,\chi;\e)=&\pi_{\T}\Phi^-(r^-(\T,\chi;\e),\T,(\zeta_\e,0);\e)-\pi_{\T}\Phi^+(r^+(\T,\chi;\e),\T+2\si_v,(\zeta_\e,0);\e)\vspace{0.2cm}\\
=&r^-(\T,\chi;\e)-r^+(\T,\chi;\e)-2\si_v\vspace{0.2cm}\\
=&\e(r_1^-(\T,\chi)-r_1^+(\T,\chi))+\CO(\e^2),\vspace{0.3cm}\\

\CH_2(\T,\chi;\e)=&\pi_{x}\Phi^-(r^-(\T,\chi;\e),\T,(\zeta_\e,0);\e)-\pi_{x}\Phi^+(r^+(\T,\chi;\e),\T+2\si_v,(\zeta_\e,0);\e)\vspace{0.2cm}\\
=&\xi_1^-(r^-(\T,\chi;\e),\T,(x_v+\e \chi,0);\e)\vspace{0.2cm}\\
&-\xi_1^+(r^+(\T,\chi;\e),\T+2\si_v,(x_v+\e \chi,0);\e)+\CO(\e^2).\vspace{0.3cm}\\
\end{array}
\]
From the construction of $\CH$ it is clear that a crossing $2\si_v$--periodic solution of system  \eqref{ps1} exists if, and only if, we find 
$(\T_\e,\chi_\e)\in K$ such that $\CH(\T_\e,\chi_\e;\e)=(0,0)$. 
To compute $\CH$ we proceed analogously to the proof of statement {\bf (a)} of Theorem \ref{tb}, but now using the expressions just obtained for 
$r_1^{\pm}$ and again Lemma \ref{fl} with $z_0=p_v$, and $z_1=(\chi,0)$, obtaining 
\[
\begin{array}{l}
\CH_1(\T,\chi;\e)=-\e\left(\dfrac{2\chi}{F_1(p_v)}+\dfrac{g_{\T}}{F_2(q_v)}\right)+\CO(\e^2),\vspace{0.2cm}\\

\CH_2(\T,\chi;\e)=-\e \dfrac{M(\T,x_v)}{F_2(q_v)}+\CO(\e^2).
\end{array}
\]
We define 
$$
\widetilde\CH(\T,\chi;\e)=-\frac{F_2(q_v)}{\e}\CH(\T,\chi;\e).
$$ 
By hypothesis there exists $\T^*\in\s^1$ such that $M(\T^*,x_v)=0$ and $(\p M\p\T)$ $(\T^*,x_v)\neq 0$. 
Therefore, for $\chi^*=-F_1(p_v)g_{\T^*}/(2F_2(q_v))$ we get
\[
\begin{array}{RL}
\det\left(\dfrac{\p\widetilde\CH}{\p(\T,\chi)}(\T^*,\chi^*,0)\right)=&\left|\begin{array}{cc} ^\# &\dfrac{2F_2(q_v)}{F_1(p_v)}\\\dfrac{\p M}{\p\T} (\T^*,x_v)&0 
\end{array}\right|\vspace{0.2cm}\\
=&-\dfrac{2F_2(q_v)}{F_1(p_v)}\dfrac{\p M}{\p\T}(\T^*,x_v)\neq 0.
\end{array}
\]
Applying the {\it implicit function theorem} it follows that, for $\e>0$ sufficiently small, there exist $\T_\e=\T^*+\CO(\e)\in(a,b)$ and 
$\chi_\e=\chi^*+\CO(\e)$ such that 
$\CH(\T_\e,\chi_\e;\e)=\widetilde\CH(\T_\e,\chi_\e;\e)=0$. 
Furthermore, $(\T_\e,\chi_\e)\in K$. 
Indeed, the hypothesis {\bf (b)} of the theorem, which reads 
\[
g_{\T^*}<\dfrac{2F_2(q_v)}{F_1(p_v)\dfrac{\p F_2}{\p x}(p_v)}\max\big\{G^{\pm}_2(\T^*,p_v)\big\}
\]
and \eqref{fold} implies that $\chi^*<\min\{\nu^{\pm}(\T^*)\}$. 
Consequently, $\chi_\e <\min\{\nu^{\pm}(\T_\e)\}$ for $\e>0$ small enough.
This concludes the proof of statement {\bf (b)} of Theorem \ref{tb}. $\Box$

\section{Conclusions and further directions}\label{sec:conc}

In this paper, we have considered a reversible planar  Filippov system $Z_0$ having a simple two-fold cycle $\CS$.
 The reversibility property forces $\CS$ always to be the boundary of a period annulus $\mathcal{A}$ of crossing periodic orbits.  Our main goal consisted in understanding how such a simple two-fold cycle $\CS$  unfolds under small non-autonomous periodic perturbations $Z_{\e}$ of $Z_0.$

As usual, the perturbation $Z_{\e}$ was assumed to be periodic with the same period of some of the periodic orbits in $\CS\cup\CA$. 
Then, generic conditions were provided  guaranteeing the persistence of such a periodic solution (see Theorems A and B). 
The construction of a suitable displacement function and its related Melnikov function was the central tools behind our study.  
For periodic solutions bifurcating from the period annulus $\CA,$ the Melnikov function was obtained, as usual, by expanding such a displacement function around $\e=0$.  
However, this approach fails when facing sliding dynamics, which appears, for instance, in the unfolding of two-fold singularities. Hence, the main novelty of this study consisted in developing a procedure for detecting the existence of sliding and crossing periodic solutions bifurcating from the simple two-fold connection $\CS,$ where the sliding dynamics must be taken into account. In particular, the detection of sliding periodic solutions is rather different, because regular perturbations of a Filippov system produce singular perturbation problems in the sliding dynamics. 
Accordingly,  tools from singular perturbation theory had to be employed.

The study of global phenomena in Filippov systems, specially polycycles such as simple two-fold cycles, is rather recent (see, for instance, \cite{AGN,NT,NTZ}). The procedure that we have developed in this paper can be applied for a wide range of polycycles in reversible Filippov systems. For instance, polycycles formed by trajectories containing several two-fold singularities. Allowing more tangential singularities increases the degeneracy of the problem, and one could certainly expect a richer dynamics bifurcating from the polycycle.

Another possible issue is to consider higher dimensional systems, such as generic cusp-cusp singularities in 3D reversible Filippov systems. If the fixed set of the involution coincides with the switching manifold, then such a system has a topological cylinder foliated by simple two-fold connections. Thus, the ideas developed in this paper could be applied for studying the bifurcation of crossing and sliding periodic solutions from this cylinder.

\section*{Acknowledgments}

We thank the referees for their comments and suggestions which helped us to improve the presentation of this paper.

DDN has been partially supported by FAPESP grants 2018/16430-8 and 2019/10269-3, and by CNPq grant 306649/2018-7. 
TMS has been supported by the Spanish MINECO-FEDER Grant PGC2018-098676-B-100 (AEI/FEDER/UE) and the Catalan grant 2017SGR1049, and  the Catalan Institution for Research and Advanced Studies via an ICREA Academia Prize 2019. This material is based upon work supported by the National Science Foundation under Grant No. DMS-1440140 while TMS was in residence at the Mathematical Sciences Research Institute in Berkeley, California, during the Fall 2018 semester. 
IOZ has been partially supported by FAPESP grant 2013/ 21078-8. 
DDN, MAT, and IOZ have been partially supported by FAPESP grant 2018/ 13481-0. DDN and IOZ have been partially supported by CNPq grant 438975/ 2018-9.
\bibliographystyle{abbrv}
\bibliography{references.bib}

\end{document}